\documentclass[12 pt]{article}
\usepackage{graphicx} 
\usepackage{amsmath}
\usepackage{bm}
\usepackage{amsthm}
\usepackage[english]{babel}
\usepackage{hyperref}
\usepackage{geometry}
\usepackage{bbold}
\usepackage{stmaryrd}
\usepackage[all]{xy}
\usepackage{tikz}
\usepackage{pgf,tikz}
\usepackage{mathrsfs}
\usepackage{MnSymbol}
\usetikzlibrary{shapes}
\usetikzlibrary{arrows.meta}
\usetikzlibrary{graphs} 
\usetikzlibrary{graphs,quotes} 
\usetikzlibrary{graphs,shapes.geometric}
\usetikzlibrary{decorations.markings}
\tikzstyle directed=[postaction={decorate,decoration={markings,
		mark=at position .65 with {\arrow{latex}}}}]
\numberwithin{equation}{section}
\newtheorem{proposition}[equation]{Proposition}
\newtheorem{theorem}[equation]{Theorem}
\newtheorem{corollary}[equation]{Corollary}
\newtheorem{lemma}[equation]{Lemma}
\newtheorem{claim}[equation]{Claim}

\newtheorem{intro}{Theorem}
\newtheorem{introth}[intro]{Theorem}

\theoremstyle{definition}
\newtheorem{definition}[equation]{Definition}

\newtheorem{remark}[equation]{Remark}
\newtheorem{example}[equation]{Example}
\newenvironment{cproof}{\begin{proof}[Proof of the
        claim]}{\end{proof}}

\newcommand\Sub{\mathrm{Sub}}
\newcommand\Sym{\mathrm{Sym}}
\newcommand\Stab{\mathrm{Stab}}
\newcommand\Ph{\mathrm{Ph}}
\newcommand\dom{\mathrm{dom}}
\newcommand\rng{\mathrm{rng}}
\newcommand\trg{\mathbf{t}}
\newcommand\src{\mathbf{s}}
\newcommand\var{\mathbf{u}}
\newcommand\BS{\mathrm{BS}}
\newcommand\Sch{\mathrm{Sch}}
\newcommand\image{\mathrm{Im}}
\newcommand\id{\mathrm{id}}
\newcommand\M{\mathrm{M}}
\newcommand\C{\mathrm{C}}

\newcommand\B{\mathrm{B}}
\newcommand\D{\mathrm{D}}

\newcommand\Spec{\mathrm{Spec}}
\newcommand\rk{\mathrm{rk}}
\newcommand\diag{\mathrm{diag}}
\newcommand\rA{\mathrm{A}}
\newcommand\rP{\mathrm{P}}
\newcommand\GL{\mathrm{GL}}
\newcommand\SL{\mathrm{SL}}
\newcommand\Z{\mathbb{Z}}
\newcommand\F{\mathbb{F}}
\title{Dynamics on the perfect kernel of non-amenable higher rank generalized Baumslag-Solitar groups}
\author{Sasha Bontemps}
\date{\today}
\newenvironment{acknowledgements}{%
	
	\begin{abstract}
	}{%
	\end{abstract}
}
\begin{document}
	
	\maketitle
	
	\begin{abstract}
		In this article, we study the space of subgroups of non-amenable generalized Baumslag-Solitar groups (GBS groups) of rank $d$, that is, groups acting cocompactly on an oriented tree with vertex and edge stabilizers isomorphic to $\mathbb{Z}^d$. Our results generalize the study of Baumslag-Solitar groups, and of GBS groups of rank $1$. We give an explicit description of the perfect kernel of a non-amenable GBS group $G$ of rank $d$ and show the existence of a partition of the perfect kernel into a countably infinite set of pieces which are invariant under the action by conjugation of $G$, and such that each piece contains a dense orbit.
	\end{abstract}
	
	{
		\small	
		\noindent\textbf{{Keywords:}} higher rank generalized Baumslag-Solitar groups; space of subgroups; Schreier graphs; perfect kernel; topologically transitive actions; Bass-Serre theory.
	}
	
	\smallskip
	
	{
		\small	
		\noindent\textbf{{MSC-classification:}}	
		37B; 20E06; 20E08.
	}

	\section{Introduction}
	
	A generalized Baumslag-Solitar group (GBS group) of rank $d$ is a group that acts cocompactly on an oriented tree such that the vertex and edge stabilizers are isomorphic to $\mathbb{Z}^d$. As a consequence of Bass-Serre theory, a generalized Baumslag-Solitar group is defined by a finite iteration of HNN extensions and amalgamated free products of $\mathbb{Z}^d$ over $\mathbb{Z}^d$. 
	
	GBS groups of rank $d > 1$ are a generalization of GBS groups of rank $1$, which arise as a natural generalization of Baumslag-Solitar groups $\BS(m,n) = \langle b, t \mid tb^mt^{-1} =~b^n \rangle$. Baumslag-Solitar groups were introduced in \cite{tworelators} to give the first examples of two generated finitely presented non-Hopfian groups. GBS groups have been widely studied in relation to various properties. In \cite{levitt}, Levitt computed the minimal number of generators of a GBS group of rank $1$. He studied their automorphism groups in \cite{levittisom}. The classification up to quasi-isometry of GBS groups of rank 1 is known (see \cite{whyte}) and the classification up to measure equivalence of Baumslag-Solitar groups has been announced by the authors of \cite{measureeq}. In \cite{modular}, the authors determined which GBS groups (of arbitrary rank) are residually finite and which are LERF.
	
	To any GBS group $G$ of rank $d$ is associated its \textbf{modular homomorphism} $\Delta_G : G \to \GL_d(\mathbb{Q})$, which was introduced in \cite{button}. Starting from a vertex $v$ of the Bass-Serre tree, the conjugation by any element $g$ of $G$ induces a homomorphism $\Stab(v) \to \Stab(g \cdot v)$. As all vertex stabilizers are commensurable, this homomorphism can be seen as an element of $\GL_d(\mathbb{Q})$. This defines a homomorphism $\Delta_G^{(v)} : G \to \GL_d(\mathbb{Q})$. Two such homomorphisms $\Delta_G^{(v)}$, $\Delta_G^{(w)}$ are conjugate. Up to conjugation, this defines the modular homomorphism $\Delta_G : G \to \GL_d(\mathbb{Q})$. A GBS group $G$ is called \textbf{unimodular} if $\image\left(\det \circ \Delta_G\right) \subseteq \{1,-1\}$. As an example, one characterization of unimodularity for a GBS group of rank $1$ is the existence of an infinite cyclic normal subgroup (see \cite{levittisom}[Section 2] for instance).
	
	In this article, we will focus on the set of subgroups $\Sub(G)$ of a GBS group $G$ from a topological point of view. This means that we are more interested in the topological structure of $\Sub(G)$, seen as a closed subset of the Cantor set $\{0,1\}^{G}$, than in the algebraic properties of the subgroups of $G$. Cantor-Bendixson theory (see \cite{Kechris}) leads to a unique decomposition $\Sub(G) = \mathcal{K}(G) \sqcup C$ into a closed subspace without isolated points called the \textbf{perfect kernel} of $G$ and a countable set $C$. As the action by conjugation of $G$ induces a homeomorphism of $\Sub(G)$, the perfect kernel is $G$-invariant. We are interested in the computation of the perfect kernel and the dynamics induced by the action by conjugation of $G$ on it. 
	
	The perfect kernel of any finitely generated abelian group is empty. In \cite{abelian}, the authors classified the set of subgroups of all countable abelian groups up to homeomorphism. In \cite{lamplighter}, the authors proved that the perfect kernel of the lamplighter group $(\mathbb{Z}/p\mathbb{Z})^n \wr \mathbb{Z}$ (where $p$ is a prime number) is the set of subgroups of $\bigoplus_{\mathbb{Z}}\left(\mathbb{Z}/p\mathbb{Z}^n\right)$.
	
	If $G$ is a finitely generated group, then finite index subgroups are isolated. Thus the perfect kernel of $G$ is included in the set of infinite index subgroups of $G$. The authors of \cite{totipotent} observed that equality holds for the non-abelian finitely generated free group $\mathbb{F}_r$ on $r$ generators, and that the action of $\mathbb{F}_r$ on its perfect kernel is topologically transitive. Recall that an action of a group $G$ on a topological space $X$ is \textbf{topologically transitive} if for every non-empty open subsets $U, V \subseteq X$, there exists $g \in G$ such that $gU \cap V \neq \emptyset$. In the case where $X$ is Polish, this is equivalent to the existence of a dense orbit. The authors of \cite{AG} extended this result to a large class of groups acting on trees. They proved that the perfect kernel of a finitely generated group $G$ with infinitely many ends is also equal to the set of infinite index subgroups, and that the action by conjugation is topologically transitive on the perfect kernel as soon as $G$ does not contain any non-trivial finite normal subgroup. More generally, they proved that for any finitely generated group $G$ that acts (minimally and irreducibly) on a tree $\mathcal{T}$ such that the action of $G$ on $\mathcal{T}$ is acylindrical, then any subgroup $H$ of $G$ satisfying that the quotient graph $H \backslash \mathcal{T}$ is infinite belongs to the perfect kernel of $G$. In this case, they also proved that the action by conjugation of $G$ on the closure of the set of subgroups $H$ acting on $\mathcal{T}$ with infinitely many orbits of edges is topologically transitive. Recall that an action of a group $G$ on a tree $\mathcal{T}$ is \textbf{acylindrical} if there exists $R > 0$ such that the stabilizer of any path of length larger than $R$ is trivial.
	
	GBS groups of rank $d$ are typical examples of groups whose action on their Bass-Serre tree is \textit{not} acylindrical, because the stabilizer of any finite subtree of the Bass-Serre tree is isomorphic to $\mathbb{Z}^d$. The authors of \cite{solitar1} (who studied Baumslag-Solitar groups) and of \cite{bontemps} (who extended some results obtained by the aforementioned authors to GBS groups of rank $1$) observed that this leads to very different dynamics for the action by conjugation of a non-amenable GBS group $G$ of rank $1$ on its perfect kernel. More precisely, they showed that $\mathcal{K}(G) = \Sub_{[\infty]}(G)$ if and only if $G$ is not unimodular. They also described a countably infinite $G$-invariant partition of $\mathcal{K}(G)$, such that $G$ acts topologically transitively on each piece. One piece is closed and all the other ones are open (and also closed if and only if $G$ is unimodular). To obtain this decomposition, the authors of \cite{solitar1} introduced the \textbf{phenotype}, which is a $G$-invariant function $\Sub(G) \to \mathbb{N}^* \cup \{\infty\}$, and which was generalized in \cite{bontemps}. This function is explicit and encodes the decomposition of the perfect kernel. 
	
	In this article, we will extend these results to non-amenable GBS groups of an arbitrary rank $d$, that is to say, those which are neither isomorphic to $\mathbb{Z}[\rA, \rA^{-1}](\mathbb{Z}^d)~\rtimes~\mathbb{Z}$ for some $\rA \in \M_d(\mathbb{Z}) \cap \GL_d(\mathbb{Q})$ (where $\mathbb{Z}[\rA, \rA^{-1}](\mathbb{Z}^d)$ is the subgroup of $\mathbb{Q}^d$ defined by $\{\rA^ku, k \in \mathbb{Z}, u \in \mathbb{Z}^d\}$ and $\mathbb{Z}$ acts on $\mathbb{Z}[\rA, \rA^{-1}](\mathbb{Z}^d)$ by multiplication by $\rA$) nor to any amalgamated free product $\mathbb{Z}^d*_{\mathbb{Z}^d}\mathbb{Z}^d$ where both injections are defined by matrices of determinant $\pm 2$ (\textit{cf.} Proposition \ref{amenable}). Understanding the space of subgroups of amenable GBS groups is a different problem: it requires an understanding of the space of subgroups of the abelian group $\mathbb{Z}\left[\rA, \rA^{-1}\right](\mathbb{Z}^d)$. We think that this should be tackled using the results of Cornulier, Guyot and Pitsch about the space of subgroups of an abelian group (\textit{cf.} \cite{abelian}), but this lies beyond the scope of our methods, which crucially makes use of the action on the Bass-Serre tree. 
	
	More precisely, we prove the following result (\textit{cf.} Theorem \ref{kerneld}):
	
	\begin{introth}\label{computation}
		Let $G$ be a non-amenable GBS group of rank $d$ defined by a reduced graph of groups $\mathscr{H}$ and let $\mathcal{T}$ be the Bass-Serre tree of $\mathscr{H}$. Then \[\mathcal{K}(G) = \{H \leq G \mid H \backslash \mathcal{T} \text{ \ is infinite} \}.\]
	\end{introth}
	
	We also give some sufficient conditions that depend on the modular homomorphism (see Section \ref{gbsd}) for the perfect kernel to be equal to $\Sub_{[\infty]}(G)$. Namely, we prove the following (\textit{cf.} Corollary~\ref{cor}):
	
	\begin{introth}\label{corint}
		Let $G$ be a non-amenable and non-unimodular GBS group of rank $d$ such that the finite index subgroups of $\image\left(\Delta_{G}^{(v)}\right)$ are $\mathbb{Q}$-irreducible. Then \[\mathcal{K}(G) = \Sub_{[\infty]}(G).\]
	\end{introth}
	
	Recall that a subgroup $\Gamma$ of $\GL_d(\mathbb{K})$ is called \textbf{$\mathbb{K}$-irreducible} if it does not stabilize any non-trivial proper vector subspace of $\mathbb{K}^d$.
	
	We also obtain a generalization of the main results of \cite{solitar1} and \cite{bontemps} (\textit{cf.} Equation \ref{dynpartperfker} and Theorem \ref{dynamic} if $G$ is not a semidirect product, and Theorem \ref{dynsemidir} otherwise):
	
	\begin{introth}\label{decomposition}
		Let $G$ be a non-amenable GBS group of rank $d$. There exists a countably infinite partition of the perfect kernel of $G$ into $G$-invariant pieces that contain dense orbits. 
	\end{introth}
	
	This implies in particular that the action is topologically transitive on each piece. 
	
	We also investigate the topology of the pieces that appear in this decomposition, which is slightly different depending on whether $G$ is a semidirect product $\mathbb{Z}^d \rtimes \mathbb{F}_r$ (\textit{cf.} Theorem \ref{dynsemidir} -- keeping the notations of this theorem, the pieces will be the sets $\mathcal{P}_{\mathscr{C}} \smallsetminus \mathcal{D}_{\mathscr{C}}$ and the orbits in the countable sets $\mathcal{D}_{\mathscr{C}}$ under the $G$-conjugation, for $\mathscr{C}$ varying in $Conj(G_v)$) or not (\textit{cf.} Proposition \ref{topopiece} -- keeping the notations of this proposition, the pieces of the decomposition of the perfect kernel will be given by the sets $\mathcal{K}(G) \cap \Ph_{\mathscr{H},v}^{-1}(\pi_v(\Lambda))$, for $\Lambda$ varying in $\Sub(\mathbb{Z}^d)$). However, as all these pieces of this partition are $F_{\sigma}$ subsets, they turn out to be $G_{\delta}$ subsets (hence Polish). 
	
	The main challenge raised by the generalization of the results of \cite{bontemps} to GBS groups of arbitrary rank is that, while injections of the edge groups into the vertex groups were simply encoded by non-zero integers in the rank $1$ case, they are now encoded by matrices. Understanding the space of subgroups of a higher rank GBS group thus requires a more subtle analysis of the action of the group \textit{via} the modular homomorphism on the set of vector subspaces of $\mathbb{Q}^d$, and a more careful analysis of partial actions. In particular, while the decomposition of Theorem~\ref{decomposition} was explicit in the case of GBS groups of rank $1$ (and encoded by the ``phenotype'', see \cite[Definition 5.2 and 5.5, Theorem 6.6]{bontemps}), it is now encoded by an abstract equivalence relation on the set of subgroups of $\Z^d$. Moreover, we do not know whether the topological transitivity can be strengthened into high topological transitivity (while the action on each piece was proved to be highly topologically transitive in \cite[Theorem 2]{bontemps}). Also notice that the need to provide a separate analysis for semidirect products $\Z^d \rtimes \F_r$ is a conceptual difference with the case of GBS groups of rank $1$ (see Section~\ref{semidirect}).
	
	The paper is organized as follows. Given a graph of groups $\mathscr{H}$ of fundamental group $G$, we extend the notion of $\mathscr{H}$-preactions and of $\mathscr{H}$-graphs in Section \ref{generalsetting}. They were introduced in \cite{hightrans} in the case where $G$ is an amalgamated free product or an HNN-extension (\textit{i.e.} if $\mathscr{H}$ consists of a single edge), and adapted in \cite{solitar1} in the case of Baumslag-Solitar groups and in \cite{bontemps} in the case of GBS groups of rank 1. These graphs can be thought of as a generalization of Schreier graphs in the case of free groups: any subgroup $H$ of a free group $\mathbb{F}_r$ can be encoded by a covering of the bouquet of $r$ circles (which is the Schreier graph of the action $(H \backslash \mathbb{F}_r, H) \curvearrowleft \mathbb{F}_r$, with respect to the standard generating set). When $H$ has infinite index, this covering is infinite, and perturbing this graph very far from the origin amounts to building a subgroup of $\mathbb{F}_r$ that is close to $H$. In Section \ref{perfectkernel}, we prove Theorem \ref{computation} and Theorem~\ref{corint}. Finally, in Section \ref{dynamicalpartition}, we show the existence of the decomposition of Theorem \ref{decomposition} and investigate the topology of the pieces. This gives rise to a natural generalization of the phenotype defined in \cite{solitar1} and in \cite{bontemps}. However, we do not know if this decomposition is still explicit in this wider context, and we do not know if our arguments can be used to prove \textit{high} topological transitivity results as in \cite{solitar2} for Baumslag-Solitar groups and in \cite{bontemps} for GBS groups of rank $1$. 

		\begin{acknowledgements}
		I am very grateful to my PhD advisor Damien Gaboriau for his support and for insightful conversations on this subject. I also warmly thank Mart\'in Gilabert Vio for one particularly fruitful discussion. I thank Rémi Coulon and Ashot Minasyan for their careful reading and valuable suggestions. Finally, I warmly thank the anonymous referee for their cautious reading that significantly improved the quality of this text. This work was supported by a CDSN from ENS de Lyon.
	\end{acknowledgements}
	
	\tableofcontents
	
	\section{Preliminaries and notations}
	
	We denote by $\mathcal{P}$ the set of prime numbers in $\mathbb{N}$. For every integer $N$ and $p \in \mathcal{P}$, we denote by $|N|_p$ the $p$-adic valuation of $N$, that is, the largest $n \in \mathbb{N}$ such that $p^n$ divides $N$. By ``countable'' we mean finite or in bijection with $\mathbb{N}$. Given a group $G$, we denote by $\Sub(G)$ the set of subgroups of $G$ and by $\Sub_{[\infty]}(G)$ the subset of $\Sub(G)$ that consists of infinite index subgroups. If $H \leq G$ is a subgroup, we denote by $[H]_{G}$ the $G$-conjugacy class of $H$. For any $d \in \mathbb{N}^*$, we denote by $\mathcal{L}(\mathbb{Z}^d)$ the set of lattices of $\mathbb{Z}^d$, \textit{i.e.} the set of finite index subgroups of $\mathbb{Z}^d$. If $\mathbb{K}$ is a field and $\M \in \M_d(\mathbb{K})$, one denotes by $\Spec_{\mathbb{K}}(\M)$ the spectrum of $\M$ in $\mathbb{K}$, \textit{i.e.} the set of elements $\lambda \in \mathbb{K}$ such that $\det(\M-\lambda I_d) = 0$. If $\Gamma$ is a subgroup of $\mathrm{GL}_d(\mathbb{K})$, we say that $\Gamma$ is $\mathbb{K}$-irreducible if $\Gamma$ does not stabilize any non-trivial vector subspace of $\mathbb{K}^d$. Likewise, an element $\gamma \in \mathrm{GL}_d(\mathbb{K})$ is $\mathbb{K}$-irreducible if it stabilizes no non-trivial vector subspace of $\mathbb{K}^d$.
	
	\subsection{Graphs}
	
	We refer to \cite{bontemps}, Section 2 for the notations and definitions around graphs and Schreier graphs. We add the following terminology: given a graph $\mathscr{H}$, an element of $\mathcal{E}(\mathscr{H}) \times \{\src, \trg\}$ is called an \textbf{half-edge}. The \textbf{inferior} half-edge of $e \in \mathcal{E}(\mathscr{H})$ is $(e, \src)$ and its \textbf{superior} half-edge is $(e, \trg)$. Given a graph $\mathscr{H}$ and a spanning tree $\mathscr{T}$ of $\mathscr{H}$, for any vertices $u, v \in \mathcal{V}(\mathscr{H})$, we denote by $[u; v]_{\mathscr{T}}$ the unique edge path in $\mathscr{T}$ that connects $u$ to $v$.
	
	\subsection{Space of subgroups of a countable group}
	
	Let $G$ be an infinite countable group. Endowed with the $\textbf{Chabauty topology}$, the set of subgroups $\Sub(G)$ of $G$ is a closed subspace of the Cantor set $\{0,1\}^{G}$. An explicit basis of open sets is given by the following clopen sets: \[\mathcal{V}(O,I) = \{H \in \Sub(G), H \cap O = \emptyset \text{ \ and \ } I \subseteq H  \}\]
	for any finite subsets $O,I \subseteq G$. 
	
We will make use of the following lemma, that describes the topology of some subsets of $\Sub(G)$.

	\begin{lemma}\label{topgen}
		Let $G$ be a countable group and let $G_0$ be any subgroup of $G$. Then, for any subgroup $H_0 \leq G_0$, the set \[\left\{H \leq G \mid H \cap G_0 = H_0\right\}\] is closed in $\Sub(G)$. If moreover the group $H_0$ is finitely generated and has finite index in $G_0$, then it is also open.
	\end{lemma}

	\begin{proof}
		One has \[\left\{H \leq G \mid H \cap G_0 = H_0\right\} = \bigcap_{(h,g) \in H_0 \times \left(G_0 \setminus H_0\right)}\left\{H \leq G \mid h \in H \text{ \ and \ } g \notin H\right\}\]
		which is closed as an intersection of basic clopen subsets of $\Sub(G)$.
		
		If $H_0$ is finitely generated and has finite index in $G_0$, let $\{h_1, ..., h_n\}$ be a finite generating set of $H_0$ and let us write $G_0/H_0 = \{H_0, g_1H_0, ..., g_mH_0\}$. We then have \[\left\{H \leq G \mid H \cap G_0 = H_0\right\} = \bigcap_{(i,j) \in \llbracket 1, n \rrbracket \times \llbracket 1, m \rrbracket}\left\{H \leq G \mid h_i \in H \text{ \ and \ } g_j \notin H\right\}\]
		which is open as a finite intersection of basic clopen sets.
	\end{proof}
	
	Applying Cantor-Bendixson Theorem (see \cite[Section 6, Chapter 1]{Kechris} for instance) to the Polish space $\Sub(G)$ leads to a unique decomposition $\Sub(G) = K \bigsqcup C$ where $C$ is countable and $K$ is a closed subspace of $\Sub(G)$ without isolated points. The set $K$ is called the \textbf{perfect kernel} of $G$ and denoted by $\mathcal{K}(G)$. It is the largest closed subset of $\Sub(G)$ without isolated points, or equivalently, the set of subgroups all of whose neighborhoods are uncountable.
	
	If $G$ is finitely generated, then finite index subgroups are isolated. In particular, we get the following inclusion $\mathcal{K}(G) \subseteq \Sub_{[\infty]}(G)$. The converse inclusion is true in the case of finitely generated free groups (see \cite{totipotent}[Proposition 2.1] and \cite{AG}[Corollary 5.17]): 
	\begin{proposition}\label{freegroup}
		Let $\mathbb{F}_r$ be the free group on $r$ generators ($2 \leq r \leq \infty$). Then \begin{itemize}
			\item if $r < \infty$, then $\mathcal{K}(\mathbb{F}_r) = \Sub_{[\infty]}(\mathbb{F}_r)$;
			\item $\mathcal{K}(\mathbb{F}_{\infty}) = \Sub(\mathbb{F}_{\infty})$.
		\end{itemize}
		Moreover, there exists a dense orbit for the action by conjugation of $\mathbb{F}_r$ on $\mathcal{K}(\mathbb{F}_r)$.
	\end{proposition}
	
	\begin{remark}\label{ideafreegp}
	The key point of the proof of Proposition \ref{freegroup} is the identification of subgroups of $\mathbb{F}_r = \langle (a_i)_{i \in \llbracket 1, r \rrbracket}\rangle$ with coverings of the bouquet $B_r$ of $r$ circles (which are exactly the Schreier graphs of subgroups of $\mathbb{F}_r$ with respect to the generating set $\{a_i, i \in \llbracket 1, r \rrbracket\}$). More precisely, this relies on the two following facts. Let $B$ be a (possibly infinite) graph. Then: \begin{itemize}
		\item for every covering $E \to B$ and every connected finite subgraph $K \subsetneq E$, there exists a covering $p : E' \to B$ whose degree is infinite and such that $E'$ contains a subgraph $K'$ which is isomorphic to $K$ as a labeled graph (this allows to obtain the aforementioned explicit description of the perfect kernel);
		\item for every coverings $E_i \to B$ ($i \in \mathbb{N}$), given any connected finite subgraph $K_i \subsetneq E_i$ (for every $i \in \mathbb{N}$), there exists a covering $p : E \to B$ whose degree is infinite and such that $E$ contains disjoint subgraphs $(K_i')_{i \in \mathbb{N}}$ such that $K_i'$ and $K_i$ are isomorphic as labeled graphs for every $i \in \mathbb{N}$ (this allows to build a dense orbit).
	\end{itemize}
	This point of view will be useful in the study of the action by conjugation of a semidirect product $\mathbb{Z}^d \rtimes \mathbb{F}_r$ on its perfect kernel (\textit{cf.} Section \ref{semidirect}). 
	\end{remark}
	\subsection{Graphs of groups}
	
	In this section, we recall the fundamentals of Bass-Serre theory. We refer to \cite{serre} for more details. A \textbf{graph of groups} is an oriented graph $\mathscr{H}$ equipped with a collection of \textbf{vertex groups} $G_v, v \in \mathcal{V}\left(\mathscr{H}\right)$, a collection of \textbf{edge groups} $G_e, e \in \mathcal{E}\left(\mathscr{H}\right)$ such that $G_e = G_{\overline{e}}$ for every edge $e \in \mathcal{E}\left(\mathscr{H}\right)$ and, for $\var \in \{\src, \trg\}$, injective homomorphisms $\iota_{e,\var}: G_e \hookrightarrow G_{\var(e)}$ such that $\iota_{e, \src} = \iota_{\overline{e}, \trg}$ for every edge $e$.  
	
	The \textbf{fundamental group} of a graph of groups $\mathscr{H}$ is defined by the following presentation: let us fix a spanning tree $\mathscr{T}$ in $\mathscr{H}$. Denote by $\{t_e, e \in \mathcal{E}\left(\mathscr{H}\right)\}$ a generating set of the free group $\mathbb{F}_{|\mathcal{E}\left(\mathscr{H}\right)|}$ of rank $|\mathcal{E}\left(\mathscr{H}\right)|$ and define \begin{equation}\label{presd} G = \left(*_{v \in \mathcal{V}\left(\mathscr{H}\right)}G_v * \mathbb{F}_{|\mathcal{E}\left(\mathscr{H}\right)|}\right) / \left\llangle \left(t_e\iota_{e,\trg}(x)t_e^{-1}\iota_{e, \src}(x)^{-1}\right)_{(e,x) \in \mathcal{E}\left(\mathscr{H}\right) \times G_e},  (t_et_{\overline{e}})_{e \in \mathcal{E}\left(\mathscr{H}\right)}, (t_e)_{e \in \mathcal{E}\left(\mathscr{T}\right)} \right\rrangle \end{equation}
	The isomorphism class of the group $G$ defined as above does not depend on the choice of the spanning tree (\textit{cf.} Proposition 20 in \cite{serre}, Section 5.1).
	
	There exists a (unique up to unique isomorphism) oriented tree $\mathcal{T}$, called \textbf{Bass-Serre tree} of $\mathscr{H}$ on which $G$ acts without inversion with quotient $\mathscr{H}$ and such that there exist sections $\mathcal{V}\left(\mathscr{H}\right) \to \mathcal{V}\left(\mathcal{T}\right)$ and $\mathcal{E}\left(\mathscr{H}\right) \to \mathcal{E}\left(\mathcal{T}\right)$ (which we denote by $v \to \tilde{v}$ and $e \to \tilde{e}$ respectively) of the projection $\pi : \mathcal{T} \to \mathscr{H}$ satisfying the following conditions: \begin{equation}\label{vertex}
		\Stab(\tilde{v}) = G_v \ \forall v \in \mathcal{V}(\mathscr{H}).
	\end{equation}
	\begin{equation}\label{edge}
		\Stab(\tilde{e}) = G_e \ \forall e \in \mathcal{E}(\mathscr{H}). \end{equation}

    More precisely (\textit{cf.} \cite{serre}[Section 5.3]), the set of vertices of $\mathcal{T}$ is \[\mathcal{V}(\mathcal{T}) = \bigsqcup_{v \in \mathcal{V}(\mathscr{H})}G/G_v,\] and its set of edges is \[\mathcal{E}(\mathcal{T}) = \bigsqcup_{e \in \mathcal{E}(\mathscr{H})}G/G_e.\]
        
	Conversely, any group action $G \curvearrowright \mathcal{T}$ without inversion is obtained by this construction (\textit{cf.} \cite[Section 5]{serre}). 
	
	\section{\texorpdfstring{$\mathscr{H}$-preactions and $\mathscr{H}$-graphs}{$\mathscr{H}$-preactions and $\mathscr{H}$-graphs}}
	
	In this section we generalize the interpretation of graphs of subgroups as ``blown up and shrunk'' Schreier graphs obtained for HNN-extensions or amalgamated free products in \cite{hightrans}, for Baumslag-Solitar groups in \cite{solitar1} and for rank $1$ GBS groups in \cite{bontemps}.

 This gives us a tool to approximate some subgroups of iterated HNN-extensions and amalgamated free products, that we will apply to GBS groups. 
	
	\subsection{General setting}\label{generalsetting}

In this section, we fix a graph of groups $\mathscr{H}$ endowed with a spanning tree $\mathscr{T}$ and we denote by $G$ the fundamental group of $\mathscr{H}$, defined by Presentation \eqref{presd}. To any subgroup of $G$, we will associate a ``$\mathscr{H}$-graph'', which is a labeled graph that satisfies some combinatorial conditions. It will reduce the problem of approximating a subgroup of $G$ to the one of approximating its $\mathscr{H}$-graph.

First, we introduce the notion of \textbf{$\mathscr{H}$-preaction} of $G$. Informally, this is a collection of partial bijections (each of these corresponding to an edge generator or an element of a vertex group in the presentation \eqref{presd}) such that the partial bijections associated to generators of a vertex group $G_v$ define a genuine $G_v$-action. The following definition generalizes the notion of an $\mathcal{H}$-preaction introduced for GBS groups of rank $1$ in \cite[Definition 3.1]{bontemps}.

	\begin{definition}
		An $\mathscr{H}$-preaction on a countable set $X$ is a collection of (possibly non-transitive) right $G_v$-actions $\alpha_v$ defined on subsets $D_v$ of $X$ for every $v \in \mathcal{V}(\mathscr{H})$ (\textit{i.e.} homomorphisms $\alpha_v : G_v \to \Sym(D_v)$) and of partial bijections $\beta_e : \dom(\beta_e) \to \rng(\beta_e)$ for every $e \notin \mathcal{E}(\mathscr{T})$ satisfying the following conditions: \begin{enumerate}
			\item for every $e \in \mathcal{E}(\mathscr{H}) \setminus \mathcal{E}(\mathscr{T})$, \begin{itemize}
				\item $\dom(\beta_e) \subseteq D_{\src(e)}$;	
				\item $\dom(\beta_e)$ is $\alpha_{\src(e)}(\iota_{e, \src}(G_e))$-stable;
			\item $\dom(\beta_{\overline{e}}) = \rng(\beta_e)$ and $\beta_{\overline{e}} = \beta_e^{-1}$;
		
			\end{itemize}
			\item for every $e \in \mathcal{E}\left(\mathscr{T}\right)$, for every $g \in G_e$ and $x \in D_{\src(e)} \cap D_{\trg(e)}$, one has \[x \cdot \alpha_{\src(e)}\left(\iota_{e, \src}(g)\right) = x \cdot \alpha_{\trg(e)}\left(\iota_{e, \trg}(g)\right);\]
			\item for every $e \in \mathcal{E}(\mathscr{H}) \setminus \mathcal{E}(\mathscr{T})$, for every $g \in G_e$ and $x \in \dom(\beta_e)$, one has \[x \cdot \alpha_{\src(e)}\left(\iota_{e, \src}(g)\right) \cdot \beta_e = x \cdot \beta_e \cdot \alpha_{\trg(e)}\left(\iota_{e, \trg}(g)\right);\]
			\item for every vertices $v, w \in \mathcal{V}(\mathscr{H})$, for every vertex $u \in [v;w]_{\mathscr{T}}$, one has \[D_v \cap D_w \subseteq D_u.\]
		\end{enumerate}
	\end{definition}
	
	As an example, any $G$-action is an $\mathscr{H}$-preaction.
	
	We say that an $\mathscr{H}$-preaction $\alpha' = \{D_v' \curvearrowleft^{\alpha_v'} G_v, \beta_e': v \in \mathcal{V}(\mathscr{H}), e \in \mathcal{E}(\mathscr{H}) \smallsetminus \mathcal{E}(\mathscr{T}) \}$ on a countable set $X'$ \textbf{extends} an $\mathscr{H}$-preaction $\alpha = \{D_v \curvearrowleft^{\alpha_v} G_v, \beta_e: v \in \mathcal{V}(\mathscr{H}), e \in \mathcal{E}(\mathscr{H}) \smallsetminus \mathcal{E}(\mathscr{T}) \}$ on a countable set $X$ if there exists an injection $\iota : X \hookrightarrow X'$ such that \begin{itemize}
	    \item for every $v \in \mathcal{V}\left({\mathscr{H}}\right)$, $\iota(D_v) \subseteq D_v'$, and on $D_v$ one has $\alpha'_v(g) \circ \iota = \iota \circ \alpha_v(g)$ for every $g \in G_v$;
		\item for every edge $e \in \mathcal{E}\left({\mathscr{H}}\right) \setminus \mathcal{E}(\mathscr{T})$, one has $\iota(\dom(\beta_e)) \subseteq \dom(\beta_e')$, and on $\dom(\beta_e)$ one has $\beta'_e \circ \iota = \iota \circ \beta_e$.
	\end{itemize}
	We then call $\alpha$ a \textbf{sub-$\mathscr{H}$-preaction} of $\alpha'$. 
	
	To alleviate notations, given an $\mathscr{H}$-preaction $\alpha$ defined on a countable set $X$, we will simply denote by $x \cdot g$ the element $x \cdot \alpha_v(g)$ if $x \in D_v$ and $g \in G_v$, and by $x \cdot t_e$ the element $x \cdot \beta_e$ if $x \in \dom(\beta_e)$. Likewise, we will write $\dom(t_e)$ and $\dom(\beta_e)$ interchangeably. 
	
	Given an $\mathscr{H}$-preaction $\alpha$ defined on a countable set $X$, one can define its \textbf{Schreier graph $\Sch(\alpha)$} in a similar way than for genuine actions: its set of vertices is $X$ and
	\begin{itemize}
	    \item for every $v \in \mathcal{V}(\mathscr{H})$ and every $x \in D_v$, there is an edge labeled $g$ with source $x$ and target $x \cdot g$ for every $g \in G_v$;
	    \item for every $e \in \mathcal{E}(\mathscr{H}) \smallsetminus \mathcal{E}(\mathscr{T})$ and every $x \in \dom(t_e)$, there is an edge labeled $t_e$ with source $x$ and target $x \cdot t_e$.
	\end{itemize}
	
		To any $\mathscr{H}$-preaction $\alpha$ on a pointed countable set $(X,x)$ is associated its \textbf{stabilizer $\Stab_{\alpha}(x)$}: any cycle $c$ of $\Sch(\alpha)$ based at $x$ defines an element $\varphi(c)$ of $G$, obtained by multiplying the labels of the edges of $c$. The stabilizer $\Stab_{\alpha}(x)$ is the subgroup of $G$ generated by all the elements $\varphi(c)$, for $c$ varying in the set of cycles of $\Sch(\alpha)$ based at $x$. 
	
	Now we turn to the definition of $\mathscr{H}$-graphs. 
	Informally, the $\mathscr{H}$-graph of an $\mathscr{H}$-preaction $\alpha$ is obtained from the Schreier graph of $\alpha$ in the following way: for every vertex $v \in \mathcal{V}\left(\mathscr{H}\right)$, shrink
every $G_v$-orbit $x \cdot G_v$ (with $x \in D_v$) to one vertex, and then label the obtained
vertex by the action $x \cdot G_v \curvearrowleft G_v$. As the isomorphism class of a transitive right action of a countable group on a countable set $X$ is uniquely determined by the stabilizer of any point of $X$, these labels will be $G_v$-conjugacy classes of subgroups of $G_v$. For any two vertices $x \cdot G_v$, $x' \cdot G_w$, we put some edges labeled $e \in \mathcal{E}(\mathscr{H})$ between $x \cdot G_v$ and $x' \cdot G_w$ if and only if $(\src(e), \trg(e)) = (v,w)$ and \begin{itemize}
	    \item either $e \in \mathscr{T}$ and $x \cdot G_v \cap x' \cdot G_w \neq \emptyset$;
        \item or $e \notin \mathscr{T}$ and $x \in \dom(t_e)$ and $x \cdot G_v t_e \cap x' \cdot G_w \neq \emptyset$.
	\end{itemize}

	More specifically, denoting by $s_e$ the element of $G$ equal to $t_e$ if $e \notin \mathcal{E}(\mathscr{H})$ and to $\mathrm{id}$ otherwise, the number of such edges is precisely the number of $G_e$-orbits of $x \cdot G_v $ which are sent by $s_e$ onto $x' \cdot G_w$.
    Because of the fact that vertex groups need not be commutative for the moment, we also add a label to the inferior and the superior half-edge in order to remember ``where'' do the two orbits $x \cdot G_vs_e$ and $x' \cdot G_w$ intersect.

    We now give the formal definition, which is a generalisation of the notion of $\mathcal{H}$-graph introduced in \cite[Definition 3.4]{bontemps}: 

	\begin{definition}
		 Let $\alpha$ be an $\mathscr{H}$-preaction of $G$ on a countable set $X$. One defines the $\mathscr{H}$-graph $\mathcal{G}$ of $\alpha$ as follows: 
		\begin{enumerate}
			\item its vertex set is the set of $G_v$-orbits for $v$ varying in $\mathcal{V}(\mathscr{H})$: \[\mathcal{V}(\mathcal{G}) = \bigsqcup_{v \in \mathcal{V}(\mathscr{H})}D_v/G_v;\]
			\item its edge set is $\mathcal{E}(\mathcal{G}) = \mathcal{E}^+(\mathcal{G}) \sqcup \mathcal{E}^-(\mathcal{G})$ where 
			\[\mathcal{E}^+(\mathcal{G}) = \bigsqcup_{e \in \mathcal{E}^+(\mathscr{T})}\left(D_{\src(e)} \cap D_{\trg(e)}\right)/\iota_{e, \src}(G_e) \bigsqcup \bigsqcup_{e \in \mathcal{E}^+(\mathscr{H}) \setminus \mathcal{E}(\mathscr{T})}\dom(t_e) / \iota_{e, \src}(G_e)  \]
			and 
			\[\mathcal{E}^-(\mathcal{G}) = \bigsqcup_{e \in \mathcal{E}^+(\mathscr{T})}\left(D_{\src(e)} \cap D_{\trg(e)}\right)/\iota_{e, \trg}(G_e) \bigsqcup \bigsqcup_{e \in \mathcal{E}^+(\mathscr{H}) \setminus \mathcal{E}(\mathscr{T})}\rng(t_e) / \iota_{e, \trg}(G_e)  \]
			with
			\begin{itemize}
				\item for every $e \in \mathcal{E}^+(\mathscr{T})$, for every $x \in D_{\src(e)} \cap D_{\trg(e)}$: 
				\[\src \left(x \cdot \iota_{e, \src}(G_e)\right) = x \cdot G_{\src(e)}\]
				and 
				\[\trg \left(x \cdot \iota_{e, \trg}(G_e)\right) = x \cdot G_{\trg(e)}.\]
				Moreover 
				\[ \overline{x \cdot \iota_{e, \src}(G_e)} = x \cdot \iota_{e, \trg}(G_e); \]
				\item for every $e \in \mathcal{E}^+(\mathscr{H}) \setminus \mathcal{E}(\mathscr{T})$, for every $x \in D_{\src(e)} \cap t_e^{-1}\left(D_{\trg(e)}\right)$: 
				\[\src \left(x \cdot \iota_{e, \src}(G_e)\right) = x \cdot G_{\src(e)}\]
				and 
				\[\trg \left(x \cdot \iota_{e, \trg}(G_e)\right) = x t_e \cdot G_{\trg(e)}.\]
				Moreover 
				\[ \overline{x \cdot \iota_{e, \src}(G_e)} = x t_e \cdot \iota_{e, \trg}(G_e); \]
			\end{itemize}
			\item \label{ident} each vertex $x \cdot G_v$ is labeled $\left([\Stab_{G_v}(x)]_{G_v}, v\right)$ (where $[\Stab_{G_v}(x)]_{G_v}$ denotes the conjugacy class of the subgroup $\Stab_{G_v}(x)$ in $G_v$);
			\item \label{choix} for every vertex $x \cdot G_v$, we fix an identification between $x \cdot G_v$ and the quotient set $\Stab_{G_v}(x) \backslash G_v$ (in a $G_v$-equivariant way). Each edge $(\Stab_{G_v}(x)g) \iota_{e, \src}(G_e)$ is labeled $e$ and its inferior half-edge is labeled $(\Stab_{G_v}(x)g) \iota_{e, \src}(G_e)$; in particular (applying this last condition to $\overline{e}$), it is labeled $\overline{(\Stab_{G_v}(x)g) \iota_{e, \src}(G_e) }$ at its target. 
		\end{enumerate}
	\end{definition}

    \begin{remark}
    In Item \ref{ident}, the data of $[\Stab_{G_v}(x)]_{G_v}$ is equivalent to the data of the isomorphism class of the $G_v$-action $\Stab_{G_v}(x) \backslash G_v \curvearrowleft G_v$, or equivalently, of the isomorphism class of the $G_v$-action on $x \cdot G_v$.
    \end{remark}

\begin{example}
    Let us consider the HNN-extension $G$ of some group $G_v$ over a group $G_e$ defined by the two inclusions $i : G_e \hookrightarrow G_v$ and $j : G_e \hookrightarrow G_v$. It is the fundamental group of the graph of groups defined in Figure \ref{hnn}.

    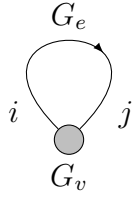
\begin{figure}[ht]
				\center
				\begin{tikzpicture}
					\node[draw,circle,fill=gray!50] (a) at (-3,0) {};
                    \draw (-3,-0.5) node {$G_v$};
					\draw[>=latex, directed] (a) to [out=135,in=45,looseness=20] node[very near start, below left]{$i$} node[very near end, below right]{$j$} node[above]{$G_e$} (a);
				\end{tikzpicture}
                \caption{Graph of groups $\mathscr{H}$ defining an HNN-extension}
				\label{hnn}
			\end{figure}
    The group $G$ inherits the following presentation \begin{equation}
        G \simeq \langle G_v, t_e \mid t_e^{-1} i(g) t_e = j(g) \rangle.
    \end{equation}

    Let us consider an $\mathscr{H}$-preaction that consists of two $G_v$-orbits $x \cdot G_v$, $y \cdot G_v$, and such that $t_e$ sends two points of $x \cdot G_v$ (say $x,x'$) that lie in two different $i(G_e)$-orbits to two other points in $y \cdot G_v$ (say $y,y'$) that lie in two different $j(G_e)$-orbits (see Figure \ref{expreac}).
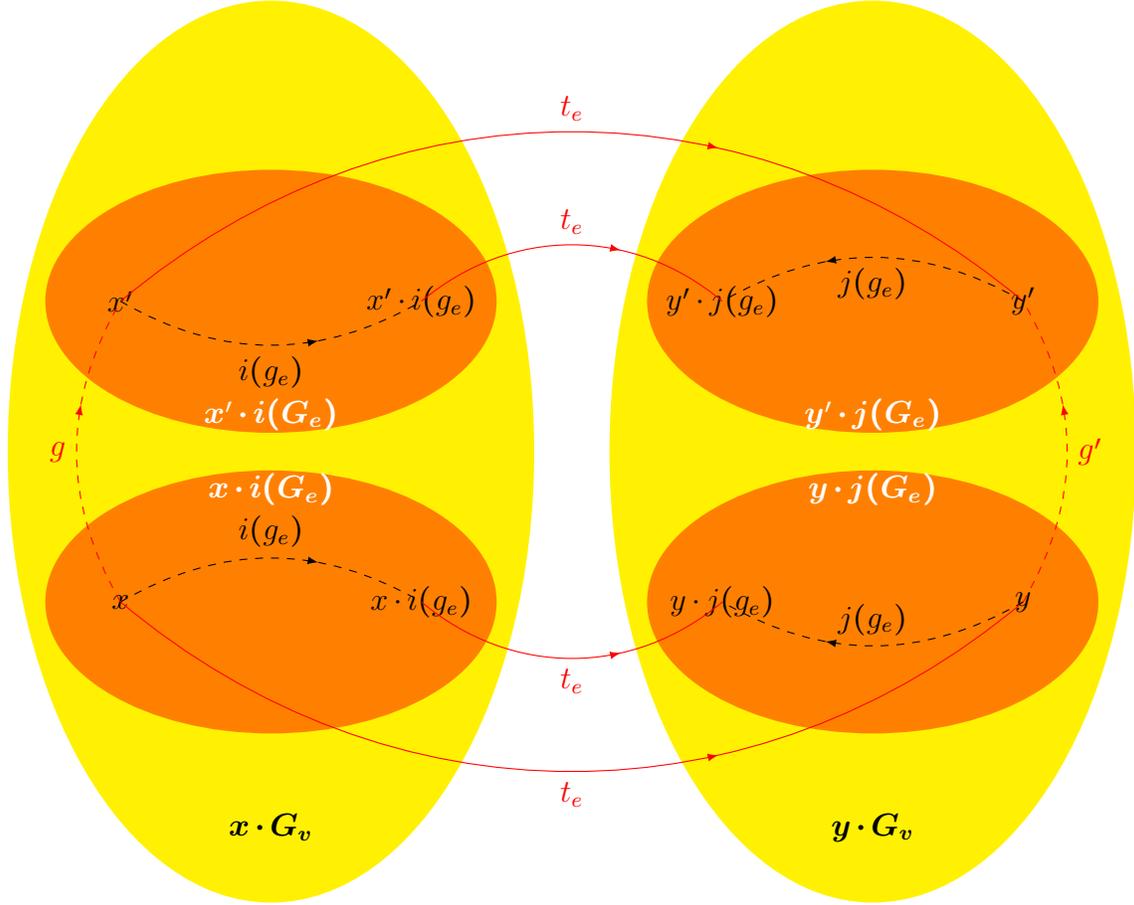
\begin{figure}[ht]
\begin{tikzpicture}
\def\fst{(0,0) ellipse (3.5 and 6)}
\def\snd{(8,0) ellipse (3.5 and 6)}
\def\ffst{(0,2) ellipse (3 and 1.75)}
\def\fsnd{(0,-2) ellipse (3 and 1.75)}
\def\sfst{(8,2) ellipse (3 and 1.75)}
\def\ssnd{(8,-2) ellipse (3 and 1.75)}
\fill[yellow] \fst \snd ;
\fill[orange] \ffst ;
\fill[orange] \fsnd ;
\fill[orange] \sfst ;
\fill[orange] \ssnd ;
\draw (-2,-2) node {$x$};
\draw (2,-2) node {$x \cdot i(g_e)$};
\draw (-2,2) node {$x'$};
\draw (2,2) node {$x' \cdot i(g_e)$};
\draw (6,-2) node {$y \cdot j(g_e)$};
\draw (10,-2) node {$y$};
\draw (6,2) node {$y' \cdot j(g_e)$};
\draw (10,2) node {$y' $};
\draw[color=white] (0,0.5) node{$\bm{x' \cdot i(G_e)}$};
\draw[color=white] (0,-0.5) node{$\bm{x \cdot i(G_e)}$};
\draw[color=white] (8,0.5) node{$\bm{y' \cdot j(G_e)}$};
\draw[color=white] (8,-0.5) node{$\bm{y \cdot j(G_e)}$};
\draw[>=latex, directed, color=red] (-2,2) to[bend left=40] node[above]{$t_e$} (10,2);
\draw[>=latex, directed, color=red] (2,2) to[bend left=40] node[above]{$t_e$} (6,2);
\draw[>=latex, directed, color=red] (-2,-2) to[bend left=-40] node[below]{$t_e$} (10,-2);
\draw[>=latex, directed, color=red] (2,-2) to[bend left=-40] node[below]{$t_e$} (6,-2);

\draw[>=latex, directed, dashed] (-2,2) to[bend left=-30] node[below]{$i(g_e)$} (2,2);
\draw[>=latex, directed, dashed] (10,2) to[bend left=-30] node[below]{$j(g_e)$} (6,2);
\draw[>=latex, directed, dashed] (-2,-2) to[bend left=30] node[above]{$i(g_e)$} (2,-2);
\draw[>=latex, directed, dashed] (10,-2) to[bend left=30] node[above]{$j(g_e)$} (6,-2);

\draw[>=latex, directed, dashed, color=red] (-2,-2) to[bend left=30] node[left]{$g$} (-2,2);
\draw[>=latex, directed, dashed, color=red] (10,-2) to[bend right=30] node[right]{$g'$} (10,2);

\draw (0,-5) node{$\bm{x \cdot G_v}$} ;
\draw (8,-5) node{$\bm{y \cdot G_v}$};
\end{tikzpicture}
\caption{Schreier graph of an $\mathscr{H}$-preaction (here $g,g' \in G_v$ and $g_e \in G_e$)}
\label{expreac}
\end{figure}

The $\mathscr{H}$-graph of the above action consists of \begin{itemize}
    \item two vertices, that correspond to the two $G_v$-orbits;
    \item two edges between those vertices, that correspond to the $i(G_e),j(G_e)$-orbits. 
\end{itemize}

This graph is represented in Figure \ref{hgraph}.
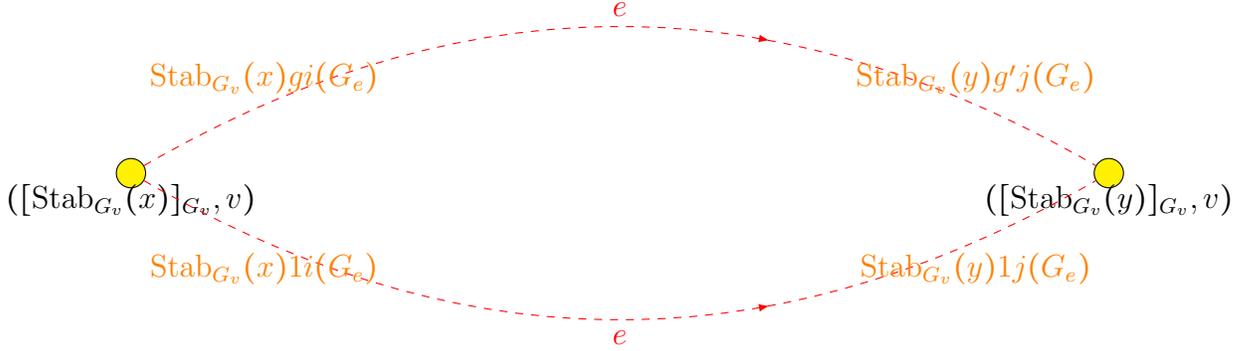
\begin{figure}[ht]
\begin{tikzpicture}
    \node[draw,circle,fill=yellow] (x) at (10,0) {};
     \node[draw,circle,fill=yellow] (y) at (23,0) {};
    \draw (x) node[below] {$([\Stab_{G_v}(x)]_{G_v},v)$};
    \draw (y) node[below] {$([\Stab_{G_v}(y)]_{G_v},v)$};
    \draw[>=latex, directed, dashed, color=red] (x) to[bend left=30] node[above]{$e$} node[above, very near start]{\color{orange} $\Stab_{G_v}(x) g i(G_e)$} node[above, very near end]{\color{orange} $\Stab_{G_v}(y) g' j(G_e)$} (y);
     \draw[>=latex, directed, dashed, color=red] (x) to[bend left=-30] node[below]{$e$} node[below, very near start]{\color{orange} $\Stab_{G_v}(x) 1_{G_v} i(G_e)$} node[below, very near end]{\color{orange} $\Stab_{G_v}(y) 1_{G_v} j(G_e)$} (y);
   \end{tikzpicture} 
   \caption{$\mathscr{H}$-graph of the $\mathscr{H}$-preaction represented in Figure \ref{expreac}}
    \label{hgraph}

\end{figure}

Notice that the chosen identifications required in Item \ref{choix} send $x$ (\textit{resp.} $y$) to the coset $\Stab_{G_v}(x)1$ (\textit{resp.} $\Stab_{G_v}(y)1$) of $\Stab_{G_v}(x) \backslash G_v$ (\textit{resp.} $\Stab_{G_v}(y) \backslash G_v$).

\end{example}

	Similarly, one defines the $\mathscr{H}$-graph of a subgroup $H$ using the correspondence between subgroups and right actions:

    \begin{definition}
        The $\mathscr{H}$-graph of a subgroup $H$ of $G$ is the $\mathscr{H}$-graph of the right action of $G$ on $H \backslash G$.
    \end{definition}

\begin{remark}\label{eqbsh}
    As observed in \cite[Section 3]{bontemps} in the case of infinite cyclic vertex and edge groups, the data of the graph of groups of $H$ is equivalent to its $\mathscr{H}$-graph: recall that the set of vertices of the Bass-Serre tree $\mathcal{T}$ associated to a graph of groups $\mathscr{H}$ of fundamental group $G$ is $\bigsqcup_{v \in \mathcal{V}(\mathscr{H})}G/G_v$, and its set of edges is $\bigsqcup_{e \in \mathcal{E}(\mathscr{H})}G/G_e$. Thus, for any subgroup $H$ of $G$, the set of vertices (\textit{resp.} of edges) of $H \backslash \mathcal{T}$ is exactly \[\bigsqcup_{v \in \mathcal{V}(\mathscr{H})}H \backslash (G/G_v) = \bigsqcup_{v \in \mathcal{V}(\mathscr{H})}(H \backslash G)/G_v,\]
    (\textit{resp.} $\bigsqcup_{e \in \mathcal{E}(\mathscr{H})}H \backslash (G/G_e) = \bigsqcup_{e \in \mathcal{E}(\mathscr{H})}(H \backslash G)/G_e)$),
    which is exactly the set of vertices (\textit{resp.} of edges) of the $\mathscr{H}$-graph of $H$ (taking the quotient $H \backslash G$ amounts to taking the Schreier graph of $H$, then taking the disjoint union over the vertices of $\mathscr{H}$ amounts to ``blowing it up'', and taking the quotient by every vertex (\textit{resp.} edge) group amounts to shrinking the orbits of vertex groups (\textit{resp.} edge groups), which amounts to constructing the vertices (\textit{resp.} edges) of the $\mathscr{H}$-graph of $H$).
\end{remark}
	
Now we want to define an abstract notion of $\mathscr{H}$-graph. To achieve this, we first prove the following lemma, which gives a combinatorial condition on the labels of the vertices of the $\mathscr{H}$-graph of a preaction.

\begin{lemma}
		Let $\mathcal{G}$ be the $\mathscr{H}$-graph of an $\mathscr{H}$-preaction $\alpha$ defined on a set $X_0$. For every $E \in \mathcal{E}(\mathcal{G})$, denoting by $e$ the label of $E$, by $([\Lambda_0]_{G_{\src(e)}}, \src(e))$ (\textit{resp.} $([\Lambda_1]_{G_{\trg(e)}}, \trg(e))$) the label of $\src(E)$ (\textit{resp.} $\trg(E)$), by $\Lambda_0 g_0 \iota_{e, \src}(G_e)$ (\textit{resp.} $\Lambda_1 g_1 \iota_{e, \trg}(G_e)$) the label of the inferior (\textit{resp.} superior) half-edge of $E$, one has \[\left[\iota_{e, \src}^{-1}\left(g_0^{-1} \Lambda_0 g_0\right)\right]_{G_e} = \left[\iota_{e, \trg}^{-1}\left(g_1^{-1}\Lambda_1g_1\right)\right]_{G_e}.\] 
	\end{lemma}
	
	\begin{proof}
		For every $e \in \mathcal{E}(\mathscr{H})$, let us define $s_e = \left\{
		\begin{array}{ll}
			t_e & \mbox{if } e \notin \mathcal{E}(\mathscr{T}) \\
			id & \mbox{otherwise}
		\end{array}
		\right.$.
		By construction, there exist $x, y \in X_0$ and $g,g' \in G_e$ such that \begin{itemize}
		    \item $\Stab_{G_{\src(e)}}(x) = \Lambda_0$;
            \item $\Stab_{G_{\trg(e)}}(y) = \Lambda_1$;
            \item $x \cdot g_0 \iota_{e, \src}(g) s_e = y \cdot g_1 \iota_{e, \trg}(g')$.
		\end{itemize}

		Hence, denoting by $X = x \cdot g_0 \iota_{e, \src}(g)$ and $Y = y \cdot g_1 \iota_{e, \trg}(g')$, one has $Y=Xs_e$, thus \begin{align*} \Stab_{G_{\trg(e)}}(Y) \cap \iota_{e, \trg}(G_e) &= 
		    \Stab_{\iota_{e, \trg}(G_e)}(Y)\\ &= \Stab_{s_e^{-1}\iota_{e, \src(G_e)}s_e}(Xs_e) \\
            &= s_e^{-1}\left(\Stab_{\iota_{e, \src}(G_e)}(X)\right)s_e \\
            &= s_e^{-1}\left(\Stab_{G_{\src(e)}}(X) \cap \iota_{e, \src}(G_e)\right)s_e \\
            &=\iota_{e, \trg}\left(\iota_{e, \src}^{-1}\left(\Stab_{G_{\src(e)}}(X)\right)\right)
		\end{align*}
		so, taking the preimage under $\iota_{e, \trg}$ ,we get \[\iota_{e, \src}^{-1}(\Stab_{G_{\src(e)}}(X)) = \iota_{e, \trg}^{-1}(\Stab_{G_{\trg(e)}}(Y))\]
		which amounts to saying that \[g^{-1}\iota_{e, \src}^{-1}\left(g_0^{-1}\Lambda_0g_0\right)g = g'^{-1}\iota_{e, \trg}^{-1}\left(g_1^{-1}\Lambda_1g_1\right)g'\]
		which proves the statement. 
	\end{proof}
	
	Using this lemma, one can extend the notion of $\mathscr{H}$-graph as follows:
	
	\begin{definition}
		An $\mathscr{H}$-graph $\mathcal{G}$ is a labeled graph satisfying the following conditions:
		\begin{itemize}
			\item every vertex is labeled $\left([\Lambda]_{G_v}, v\right)$ for some $v \in \mathcal{V}(\mathscr{H})$ and some subgroup $\Lambda$ of $G_v$;
			\item every edge $E \in \mathcal{E}(\mathcal{G})$ is labeled by an edge $e \in \mathcal{E}(\mathscr{H})$ such that $\src(E)$ is labeled $(C, \src(e))$ (for some $G_{\src(e)}$-conjugacy class $C$ of subgroups of $G_{\src(e)}$) and $\trg(E)$ is labeled $(C', \trg(e))$ (for some $G_{\trg(e)}$-conjugacy class $C'$ of subgroups of $G_{\trg(e)}$);
			\item the inferior half-edge of an edge labeled $e$ with source labeled $\left([\Lambda]_{G_v}, v\right)$ is labeled by an element of $\Lambda \backslash G_v / \iota_{e, \src}(G_e)$;
			\item every edge labeled $e$ whose inferior half-edge is labeled $\Lambda_0 g_0 \iota_{e, \src}(G_e)$ (with $\Lambda_0 \leq G_{\src(e)}$) and whose superior half-edge is labeled $\Lambda_1 g_1 \iota_{e, \trg}(G_e)$ (with $\Lambda_1 \leq G_{\trg(e)}$) satisfies \[\left[\iota_{e, \src}^{-1}\left(g_0^{-1} \Lambda_0 g_0\right)\right]_{G_e} = \left[\iota_{e, \trg}^{-1}\left(g_1^{-1} \Lambda_1 g_1\right)\right]_{G_e};\]
			\item inferior (\textit{resp.} superior) half-edges of different edges labeled $e$ (for some $e \in \mathcal{E}(\mathscr{H})$) sharing the same source (\textit{resp.} target) cannot share the same label.
		\end{itemize}
	\end{definition}
	
	A vertex $V$ of an $\mathscr{H}$-graph labeled $([\Lambda]_{G_v},v)$ is called \textbf{saturated relatively to an edge $e \in \mathcal{E}(\mathscr{H})$} if and only if $\src(e) = v$ and the set of labels of the inferior half-edges of edges labeled $e$ with source $V$ is the whole $\Lambda \backslash G_v / \iota_{e, \src}(G_e)$. The vertex $V$ is called \textbf{saturated} if it is saturated relatively to every edge with source $v$. An $\mathscr{H}$-graph is called \textbf{saturated} if all its vertices are saturated. In the case where edge groups have finite index in vertex groups, notice that this holds if and only if every vertex labeled $([\Lambda]_{G_v},v)$ has exactly $\left|\Lambda \backslash G_v / \iota_{e, \src}(G_e)\right|$ outgoing edges labeled $e$ for every edge $e \in \mathcal{E}(\mathscr{H})$ satisfying $\src(e) = v$. Finally, an $\mathscr{H}$-action is called saturated if its $\mathscr{H}$-graph is saturated, namely if it is a genuine $G$-action.
	
	In the formalism of Bass, who introduced the right notion of coverings in the setting of graphs of groups in \cite{bass}, an $\mathscr{H}$-graph $\mathcal{G}$ is the base space of an \emph{immersion} $\mathcal{G} \to \mathscr{H}$ of the graph of groups $\mathscr{H}$. This immersion is a \emph{covering} if and only if $\mathcal{G}$ is saturated, that is to say, $\mathcal{G}$ is the $\mathscr{H}$-graph of a subgroup of $G$. Notice however that the notion of coverings of graphs of groups introduced by Bass is finer than our notion of $\mathscr{H}$-graphs, as he establishes a correspondence between $\Sub(G)$ and coverings of the graph of groups $\mathscr{H}$. Here, different subgroups can share the same $\mathscr{H}$-graph: consider for instance a graph of groups $\mathscr{H}$ defined by a loop based at a vertex labeled $H$ (for some abstract group $H$), and take the edge group to be equal to $H$, and both inclusions to be the identity. The fundamental group of $\mathscr{H}$ is $H \times \Z$ and for any element $h$ of $H$, the $\mathscr{H}$-graph of the cyclic subgroup $\langle (h,1) \rangle$ consists of a single loop based at a vertex labeled $H \cap \langle (h,1) \rangle = \{\id\}$. However, the subgroups $\langle (h,1) \rangle$ are pairwise different when $h$ varies in $H$. 
	
	The three following lemmas  (\ref{completiond}, \ref{fini} and \ref{forest}) give rise to a useful tool to approximate subgroups of $G$. We show that, under some conditions, an approximation of the $\mathscr{H}$-graph of a subgroup gives rise to an approximation of the subgroup itself. More specifically, if we start from a subgroup $H$ of $G$ and any non-saturated $\mathscr{H}$-preaction $\alpha_1$ of the corresponding $\mathscr{H}$-preaction, Lemma~\ref{forest} and Lemma~\ref{completiond} will allow to extend $\alpha_1$ into a saturated $\mathscr{H}$-preaction $\beta_1$ whose $\mathscr{H}$-graph $\mathcal{G}$ contains the $\mathscr{H}$-graph $\mathcal{G}_1$ of $\alpha_1$, and such that $\mathcal{G} / \mathcal{G}_1$ is a tree. If moreover we were able to build a finite connected $\mathscr{H}$-graph that contains both a cycle $\mathcal{L}$ and $\mathcal{G}_1$ as disjoint subgraphs, then Lemma~\ref{fini}, together with Lemma~\ref{forest} and Lemma~\ref{completiond} would yield a saturated $\mathscr{H}$-preaction $\beta_2$ that extends $\alpha_1$ and whose $\mathscr{H}$-graph $\mathcal{G}'$ contains $\mathcal{G}_1$, such that $\mathcal{G}'/\mathcal{G}_1$ contains the single cycle $\mathcal{L}$. Those two preactions $\beta_1$ and $\beta_2$ will extend $\alpha_1$. Moreover, they will be non-isomorphic because the underlying graphs of their $\mathscr{H}$-graphs do not share the same homotopy type. 
	
	The proofs of these are very similar to the ones given in the case of GBS groups of rank $1$ in \cite{bontemps} (Lemma 3.3, Lemma 3.4 and Lemma 3.5). For the convenience of the reader, we adapt the main ingredients of the proofs to our wider setting.
	
	\begin{lemma}\label{completiond}
		Let $(\alpha_i)_{i \in \mathbb{N}}$ be a collection of $\mathscr{H}$-preactions whose $\mathscr{H}$-graphs $(\mathcal{G}_i)_{i \in \mathbb{N}}$ are contained in a saturated $\mathscr{H}$-graph $\mathcal{G}$ as pairwise disjoint subgraphs such that the quotient $\mathcal{G} / \left(\bigsqcup_{i \in \mathbb{N}} \mathcal{G}_i\right)$ is a tree. There exists a $G$-action $\alpha$ whose $\mathscr{H}$-graph is $\mathcal{G}$ and that extends $\alpha_i$ for every $i \in \mathbb{N}$.
	\end{lemma}
	
		\begin{lemma}\label{fini}
		Let $\mathcal{F}$ be a finite $\mathscr{H}$-graph. There exists an $\mathscr{H}$-preaction whose $\mathscr{H}$-graph is $\mathcal{F}$.
	\end{lemma}
	The proofs of Lemma \ref{completiond} and Lemma \ref{fini} rely on straightforward inductions based on the following proposition:
	
	\begin{proposition} Let $\alpha_0$ be an $\mathscr{H}$-preaction defined on a countable set $X_0$ whose $\mathscr{H}$-graph $\mathcal{G}_0$ is contained in an $\mathscr{H}$-graph $\mathcal{G}$ such that: \begin{itemize}
			\item $\mathcal{V}(\mathcal{G}) = \mathcal{V}(\mathcal{G}_0)$;
			\item $\mathcal{E}(\mathcal{G}) = \mathcal{E}(\mathcal{G}_0) \sqcup \{E\}$ for some edge $E$ such that $\src(E), \trg(E) \in \mathcal{V}(\mathcal{G}_0)$.
		\end{itemize}
		Then, there exists an $\mathscr{H}$-preaction $\alpha$ whose $\mathscr{H}$-graph is $\mathcal{G}$. Moreover, if $\src(E)$ and $\trg(E)$ belong to two different connected components of $\mathcal{G}_0$ (that is to say, $\alpha_0 = \alpha_1 \sqcup \alpha_2$ for some sub-$\mathscr{H}$-preactionss $\alpha_1$ and $\alpha_2$, and $\src(E)$ (\textit{resp.} $\trg(E)$) corresponds to a vertex orbit for $\alpha_1$ (\textit{resp.} $\alpha_2$)), the constructed preaction $\alpha$ extends both $\alpha_1$ and $\alpha_2$.
	\end{proposition}
	
	\begin{proof}
		We adapt Constructions $\rA$ and $\B$ defined in \cite[Section 3]{bontemps}. Let $e$ be the label of $E$ and let us denote by $(V_1, V_2) := (\src(E), \trg(E))$, and by $(v_1,v_2) := (\src(e), \trg(e))$. Let $([\Lambda_i]_{G_{v_i}}, v_i)$ be the label of $V_i$ (for $i \in \{1,2\}$) and let $\Lambda_1 g_1 \iota_{e, \src}(G_e)$ and $\Lambda_2 g_2 \iota_{e, \trg}(G_e)$ be the label of the inferior half-edge and of the superior half-edge of $E$, respectively. One has $[\iota_{e, \src}^{-1}(g_1^{-1}\Lambda_1 g_1)]_{G_e} = [\iota_{e, \trg}^{-1}(g_2^{-1}\Lambda_2 g_2)]_{G_e}$, \textit{i.e.} there exists $h \in G_e$ such that \begin{equation}\label{tr}\iota_{e, \src}^{-1}(g_1^{-1}\Lambda_1 g_1) = h^{-1}\iota_{e, \trg}^{-1}(g_2^{-1}\Lambda_2 g_2) h.\end{equation}
		We distinguish two cases:
		
		\paragraph{Construction A}: If $e \notin \mathcal{E}(\mathscr{T})$, then there exist $x_1,x_2 \in X_0$ such that: \begin{itemize}
			\item $x_i \in \dom(G_{v_i})$ for $i \in \{1,2\}$;
			\item $\Stab_{G_{v_i}}(x_i) = \Lambda_i$ for $i \in \{1,2\}$;
			\item $x_1 \cdot g_1 \notin \dom(t_e)$ and $x_2 \cdot g_2 \notin \rng(t_e)$.
		\end{itemize}
		We extend $t_e$ and $t_e^{-1}$ on a subset of $X:=X_0$ by defining, for every $g \in G_e$: \[\left(x_1 \cdot g_1(\iota_{e, \src}(g))\right) \cdot t_e = x_2 \cdot g_2 \iota_{e, \trg}(hgh^{-1}),\] which is well-defined by Equation \ref{tr}. The resulting $\mathscr{H}$-preaction is suitable.
		
		\paragraph{Construction B}:  If $e \in \mathcal{E}(\mathscr{T})$, then there exist $x_1,x_2 \in X_0$ such that: \begin{itemize}
			\item $x_i \in \dom(G_{v_i})$ for $i \in \{1,2\}$;
			\item $\Stab_{G_{v_i}}(x_i) = \Lambda_i$ for $i \in \{1,2\}$;
			\item $x_1 \cdot g_1 \notin \dom(G_{v_2})$ and $x_2 \cdot g_2 \notin \dom(G_{v_1})$.
		\end{itemize}
		We let \[X := X_0 / \left(x_1 \cdot g_1 \iota_{e, \src}(g) \sim x_2 \cdot g_2 \iota_{e, \trg}(hgh^{-1}) \ \forall g \in G_e\right)\] and we let $\alpha$ be the $\mathscr{H}$-preaction induced by $\alpha$ on $X$. As previously, it is well-defined by Equation \ref{tr}. By construction, the $\mathscr{H}$-graph of $\alpha$ is $\mathcal{G}$.  
	\end{proof}
	
	Finally we explain how to saturate an $\mathscr{H}$-graph in our new setting.
	
	\begin{lemma}\label{forest}
		For every (non-saturated) $\mathscr{H}$-graph $\mathcal{G}$, there exists a saturated $\mathscr{H}$-graph $\tilde{\mathcal{G}}$ that contains $\mathcal{G}$ and such that the quotient $\tilde{\mathcal{G}} / \mathcal{G}$ is an infinite forest. 
	\end{lemma}
	
	\begin{proof}
		We build an increasing sequence of $\mathscr{H}$-graphs $(\mathcal{G}_n)_{n \mathbb{N}}$ such that the union $\bigcup_{n \in \mathbb{N}}\mathcal{G}_n$ is the desired $\mathscr{H}$-graph. This works by induction using the following construction: if $V$ is a non-saturated vertex labeled $([\Lambda_0]_{G_v}, v)$, whose non-saturation is witnessed by \begin{itemize}
			\item an edge $e \in \mathcal{E}(\mathscr{H})$ with source $v$;
			\item an element $g_0 \in G_v$ such that there is no inferior half-edge labeled $\Lambda_0 g_0 \iota_{e, \src}(G_e)$ with source $V$
		\end{itemize}
		then, denoting by $w = \trg(e)$, one defines \[\Lambda_1 = \iota_{e, \trg}(\iota_{e, \src}^{-1}(g_0^{-1}\Lambda_0 g_0))\] and one builds a new vertex $W$ labeled $([\Lambda_1]_{G_w}, w)$ and a new edge $E$ labeled $e$ with source $V$ and target $W$ such that \begin{itemize}
			\item the inferior half-edge $(E, \src)$ is labeled $\Lambda_0 g_0 \iota_{e, \src}(G_e)$;
			\item the superior half-edge $(E, \trg)$ is labeled $\Lambda_1 \iota_{e, \trg}(G_e)$.
		\end{itemize}
	\end{proof}
	
	\subsection{Generalized Baumslag-Solitar groups}\label{gbsd}

In this section, we will apply the tools we introduced in the previous section to GBS groups.
	
	A GBS group of rank $d$ is the fundamental group of a finite graph of groups whose vertex and edge stabilizers are isomorphic to $\mathbb{Z}^d$. As any injective homomorphism $\mathbb{Z}^d \to \mathbb{Z}^d$ is represented by an element of $\M_d(\mathbb{Z}) \cap \GL_d(\mathbb{Q})$, a GBS group can be represented by an oriented graph $\mathscr{H}$ endowed with a function which associates an integer matrix whose determinant is non-zero to each half-edge: \[\begin{array}{ccccc}
		\M & : & \mathcal{E}\left(\mathscr{H}\right) \times \{\src, \trg\} & \to & \M_d(\mathbb{Z}) \cap \GL_d(\mathbb{Q}) \\
		& & (e, \var) & \mapsto & \M_{e, \var} \\
	\end{array}\]
	and that satisfies $\M_{\overline{e}, \trg} = \M_{e, \src}$ for every $e \in \mathcal{E}(\mathscr{H})$.
	Up to shrinking some edges, we can assume that the graph $\mathscr{H}$ is \textbf{reduced}, that is to say, the only edges $e \in \mathcal{V}(\mathscr{H})$ one of whose labels is in $\GL_d(\mathbb{Z})$ are loops. 
	
	Given a GBS group $G$ defined by a graph of groups $\mathscr{H}$, let us denote by $\mathcal{T}$ its Bass-Serre tree and let us define the \textbf{modular homomorphism} as follows (\textit{cf.} \cite{modular}): fix a vertex $v \in \mathcal{V}(\mathcal{T})$. Observe that for any $g \in G$, the group $G_{gv} \cap G_v$ has finite index in $G_v$, hence $g$ belongs to the abstract commensurator of $G_v$. Define $\Delta_{G}^{(v)}(g)$ as the equivalence class of the homomorphism $\begin{array}{ccccc}
		G_v \cap G_{g^{-1}v} & \to & G_v \cap G_{gv} \\
		h & \mapsto & ghg^{-1} \\
	\end{array}$. As $G_v$ is isomorphic to $\mathbb{Z}^d$, the homomorphism $\Delta_{G}^{(v)}$ can be identified with a homomorphism $\Delta_{G}^{(v)} : G \to \GL_d(\mathbb{Q})$. The definition of the modular homomorphism does not depend on the choice of the vertex $v$ up to conjugation by an element of $\GL_d(\mathbb{Q})$. Practically, the image of the modular homomorphism (based at some vertex $v$) is generated by the matrices $\rA_1\B_1^{-1} \cdots \rA_n \B_n^{-1}$ for every edge path $e_1,...,e_n$ labeled $(\rA_1,\B_1)$, ... ,$(\rA_n,\B_n)$ and based at $v$. In particular, the modular homomorphism of a GBS group defined by a tree of groups is trivial. A GBS group $G$ of rank $d$ is \textbf{unimodular} if $\image(\det \circ \Delta_G) \subseteq \{1,-1\}$.

    \begin{example}

Let us consider the GBS group of rank $2$ defined in Figure \ref{exgbs}.

\begin{figure}[ht]
\center
\begin{tikzpicture}
    \node[draw,circle,fill=pink] (a) at (0,0) {$v$};
     \node[draw,fill=cyan] (b) at (7,0) {$w$};
     \draw[>=latex, directed] (a) to[bend left=20] node[above, very near start]{$\begin{pmatrix}
  1 & 5\\ 
  -2 & 0
\end{pmatrix}$} node[above, very near end]{$\begin{pmatrix}
  -1 & 8\\ 
  4 & 5
\end{pmatrix}$} node[above]{$e$} (b);
    \draw[>=latex, directed, dashed, color=red] (a) to[bend left=-20] node[below, very near start]{$\begin{pmatrix}
  7 & -1\\ 
  -3 & -3
\end{pmatrix}$} node[below, very near end]{$\begin{pmatrix}
  -4 & 1\\ 
  3 & 9
\end{pmatrix}$} node[below]{$f$} (b);
\end{tikzpicture}
\caption{A graph of GBS group}
\label{exgbs}
\end{figure}
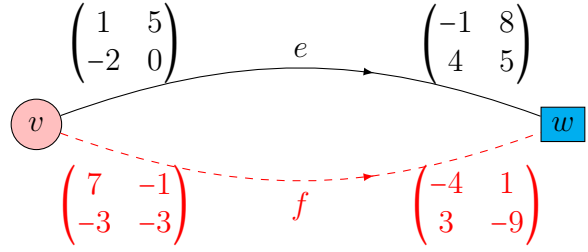
Let us choose the edge $f$ as a spanning tree in the graph represented in Figure \ref{exgbs}. Then, the presentation of $G$ associated to this choice is the following: \begin{align*}
    G \simeq &\left\langle x_v,y_v,x_w,y_w,t_e \mid [x_v,y_v]=[x_w,y_w]=1, \right.\\
    & x_v^7y_v^{-3} = x_w^{-4}y_w^3, x_v^{-1}y_v^{-3} = x_wy_w^{9},  \\
    &\left. t_e^{-1} (x_vy_v^{-2})t_e=x_w^{-1}y_w^4, t_e^{-1}x_v^5t_e = x_w^8y_w^5 \right\rangle.
\end{align*}
The modular homomorphism is trivial on the vertex generators $x_v,y_v,x_w,y_w$ and sends $t_e$ to the product $\begin{pmatrix}
  1 & 5\\ 
  -2 & 0
\end{pmatrix} \cdot \begin{pmatrix}
  -1 & 8\\ 
  4 & 5
\end{pmatrix}^{-1} \cdot \begin{pmatrix}
  -4 & 1\\ 
  3 & 9
\end{pmatrix} \cdot \begin{pmatrix}
  7 & -1\\ 
  -3 & -3
\end{pmatrix}^{-1}$: \[\Delta_G^{(v)} : \begin{array}{ccccc}
 & G   &\to  & \GL_2(\mathbb{Q}) \\
&  x_v,y_v,x_w,y_w & \mapsto & I_2 \\
& t_e & \mapsto & \frac{1}{296} \begin{pmatrix}
  -153 & -301\\ 
  46 & 342
\end{pmatrix}
\end{array}\]
        
    \end{example}
	
	In the case where the vertex stabilizers are abelian, the definition of an $\mathscr{H}$-graph simplifies as follows. An $\mathscr{H}$-graph $\mathcal{G}$ is then a labeled graph satisfying the following conditions:
	\begin{enumerate}
		\item every vertex is labeled $(\Lambda, v)$ for some $v \in \mathcal{V}(\mathscr{H})$ and some subgroup $\Lambda$ of $G_v$;
		\item every edge $E \in \mathcal{E}(\mathcal{G})$ is labeled $e$ for some $e \in \mathcal{E}(\mathscr{H})$ such that $\src(E)$ is labeled $(\Lambda_0, \src(e))$ (for some $\Lambda_0 \leq G_{\src(e)}$) and $\trg(E)$ is labeled $(\Lambda_1, \trg(e))$ (for some $\Lambda_1 \leq G_{\trg(e)}$);
		\item every edge labeled $e$ whose source (\textit{resp.} target) is labeled $(\Lambda_0, \src(e))$ (\textit{resp.} $(\Lambda_1, \trg(e))$) satisfies \[\iota_{e, \src}^{-1}\left(\Lambda_0\right) = \iota_{e, \trg}^{-1}\left(\Lambda_1\right);\]
		\item every vertex labeled $(\Lambda, v)$ has at most $\left|\Lambda \backslash G_v / \iota_{e, \src}(G_e)\right|$ outgoing edges labeled $e$ (with $\src(e) = v$) and at most $\left|\Lambda \backslash G_v / \iota_{e, \trg}(G_e)\right|$ incoming edges labeled $e$ (with $\trg(e) = v$).
	\end{enumerate}
	
	Hence for a GBS group of rank $d$ the definition of an $\mathscr{H}$-graph becomes the following: 
	
	\begin{definition}
		Let $d \geq 1$ and let $\mathscr{H}$ be a finite graph of groups all of whose vertex and edge groups are isomorphic to $\mathbb{Z}^d$. An $\mathscr{H}$-graph is a labeled graph that satisfies the three following conditions:
		\begin{enumerate}
			\item every vertex is labeled $(\Lambda, v)$ for some $v \in \mathcal{V}(\mathscr{H})$ and some subgroup $\Lambda$ of $\mathbb{Z}^d$;
			\item every edge $E \in \mathcal{E}(\mathcal{G})$ is labeled $e$ for some $e \in \mathcal{E}(\mathscr{H})$ such that $\src(E)$ is labeled $(\Lambda_0, \src(e))$ (for some $\Lambda_0 \leq G_{\src(e)}$) and $\trg(E)$ is labeled $(\Lambda_1, \trg(e))$ (for some $\Lambda_1 \leq G_{\trg(e)}$);
			\item \textbf{Transfer Equation}\label{transfer} every edge labeled $e$ whose source (\textit{resp.} target) is labeled $(\Lambda_0, \src(e))$ (\textit{resp.} $(\Lambda_1, \trg(e))$) satisfies \[\left(\M_{e, \src}^{-1}\Lambda_0\right) \cap \mathbb{Z}^d = \left(\M_{e, \trg}^{-1}\Lambda_1\right) \cap \mathbb{Z}^d;\]
			\item every vertex labeled $(\Lambda, v)$ has at most $\left|\mathbb{Z}^d / \langle \Lambda, \M_{e, \src}\mathbb{Z}^d \rangle\right|$ incident edges labeled $e$ (with $\src(e) = v$) and at most $\left|\mathbb{Z}^d / \langle \Lambda, \M_{e, \trg}\mathbb{Z}^d \rangle \right|$ incident edges labeled $e$ (with $\trg(e) = v$);
		\end{enumerate}
		It is saturated if and only if equality holds for every vertex and edge in the last item.
	\end{definition}
	
	\begin{definition}
	        An \textbf{$\mathscr{H}$-path} is an $\mathscr{H}$-graph whose underlying graph $E_1, \ldots, E_r$ is an edge path. It is said to be \textbf{reduced} if the underlying graph is reduced, namely there is no $i \in \llbracket 1, r-1 \rrbracket$ such that $E_i = \overline{E_{i+1}}$.
	\end{definition}
	
	\begin{example}
		Let $\rA = \begin{pmatrix}
			2 & 0 \\
			0 & 2
		\end{pmatrix}$ and $\B = \begin{pmatrix}
			1 & 1 \\
			1 & 4
		\end{pmatrix}$. Let us define the GBS group $G_0$ as the fundamental group $\pi_1(\mathscr{H}_0)$ of the following graph of groups defined in Figure \ref{exemple}.

		\begin{figure}[ht]
			\center
			\begin{tikzpicture}
				\node[draw,circle,fill=pink] (a) at (0,0) {};
				\node[draw,fill=orange] (b) at (3,0) {};
				\draw[>=latex, directed, color=blue] (a) to [out=135,in=45,looseness=20] node[very near start, below left]{$\rA$} node[very near end, below right]{$I_2$} (a);
				\draw[>=latex, directed, dashed, color=red] (a) to node[very near start, below left]{$\rA$} node[very near end, below right]{$\B$} (b);
			\end{tikzpicture}
			\caption{The graph of groups $\mathscr{H}_0$.}
			\label{exemple}
		\end{figure}
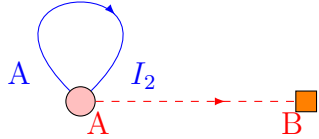

		The labeled graph represented in Figure \ref{exhgraph} is a (non-saturated) $\mathscr{H}_0$-graph.

		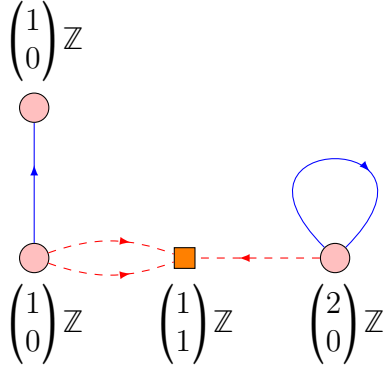
\begin{figure}[ht]
			\center
			\begin{tikzpicture}
				\node[draw,circle,fill=pink] (a) at (0,0) {};
				\draw (0.15, -0.2) node[below]{$\begin{pmatrix}
						1 \\
						0
					\end{pmatrix}\mathbb{Z}$};
				\node[draw,circle,fill=pink] (b) at (0,2) {};
				\draw (0.15,2.2) node[above]{$\begin{pmatrix}
						1 \\
						0
					\end{pmatrix}\mathbb{Z}$};
				\node[draw,fill=orange] (c) at (2,0) {};
				\draw (2.15,-0.2) node[below]{$\begin{pmatrix}
						1 \\
						1
					\end{pmatrix}\mathbb{Z}$};
				\node[draw, circle, fill=pink] (d) at (4,0) {};
				\draw (4.15,-0.2) node[below]{$\begin{pmatrix}
						2 \\
						0
					\end{pmatrix}\mathbb{Z}$};
				\draw[>=latex, directed, color=blue] (d) to [out=135,in=45,looseness=20] (d);
				\draw[>=latex, directed, dashed, color=red] (d) to (c);
				\draw[>=latex, directed, dashed, color=red] (a) to[bend left=20] (c);
				\draw[>=latex, directed, dashed, color=red] (a) to[bend left=-20] (c);
				\draw[>=latex, directed, color=blue] (a) to (b);
			\end{tikzpicture}
			\caption{An example of an $\mathscr{H}_0$-graph.}
			\label{exhgraph}
		\end{figure}

		We will focus on non-amenable GBS groups. They are characterized by the following proposition:
		
		\begin{proposition}\label{amenable}
			Let $G$ be a GBS group of rank $d$ defined by a reduced graph of groups $\mathscr{H}$. Then $G$ is amenable if and only if $\mathscr{H}$ is a single loop one of whose labels is in $\GL_d(\mathbb{Z})$, or a single edge $e$ with $\src(e) \neq \trg(e)$ both of whose labels have determinant $\pm 2$.
		\end{proposition}
		
		\begin{proof}
			Let us assume that $\mathscr{H}$ consists of a single loop $e$ labeled $(\rA, \B)$ with $\B \in \GL_d(\mathbb{Z})$. Denoting by $\M := \rA\B^{-1}$, the group $G$ is isomorphic to $\mathbb{Z}[\M, \M^{-1}]\mathbb{Z}^d \rtimes \mathbb{Z}$, where $\mathbb{Z}$ acts on $\mathbb{Z}[\M, \M^{-1}]\mathbb{Z}^d$ multiplication by $\M$. As an extension of an abelian group by an abelian group, $G$ is amenable.
			
			Now let us assume that $\mathscr{H}$ consists of an edge $e$ which is not a loop such that the labels $\rA$ and $\B$ of $e$ have determinant $\pm 2$. The Bass-Serre tree $\mathcal{T}$ of $\mathscr{H}$ is a bi-infinite line on which $G$ acts with kernel $N$ isomorphic to $\mathbb{Z}^d$. The action of $G/N$ on $\mathcal{T}$ has also a single orbit of edges, the stabilizer of any edge is trivial, and, as $\rA$ and $\B$ have determinant $\pm 2$, the stabilizer of any vertex is isomorphic to $\mathbb{Z}/2\mathbb{Z}$. Thus, Bass-Serre theory tells us that $G/N$ is isomorphic to $(\mathbb{Z}/2\mathbb{Z}) * (\mathbb{Z}/2\mathbb{Z})$, which is virtually isomorphic to $\mathbb{Z}$. Hence, $N$ and $G/N$ are amenable, so $G$ is amenable.
			
			Conversely, let us assume that $\mathscr{H}$ is neither a single loop one of whose labels is in $\GL_d(\mathbb{Z})$, nor a single edge $e$ with $\src(e) \neq \trg(e)$ both of whose labels have determinant $\pm 2$. Then the action of $G$ on its Bass-Serre tree $\mathcal{T}$ is of general type, thus $G$ contains a free group on two generators. In particular, $G$ is non-amenable. 
		\end{proof}

	\end{example}
	
	\section{\texorpdfstring{An equivalence relation on $\Sub(\mathbb{Z}^d)$}{An equivalence relation on $\Sub(\mathbb{Z}^d)$}}\label{eq}
	
	In this section, we assume that $G$ is a non-amenable GBS group which is \textit{not} the fundamental group of any graph of groups defined by a single vertex and a collection of loops labeled by invertible integer matrices. In other words, $G$ is neither amenable nor isomorphic to a semidirect product $\mathbb{Z}^d \rtimes \mathbb{F}_r$.
	
First, we introduce an equivalence relation that will give the decomposition of Theorem \ref{decomposition} (the fact that it is indeed an equivalence relation will be proved in Corollary~\ref{coreq}).	

\begin{definition}
		Let $\Lambda_0$ and $\Lambda_1$ be two subgroups of $\mathbb{Z}^d$. We say that $\Lambda_0$ and $\Lambda_1$ are $\mathscr{H}$-equivalent with respect to a vertex $v \in \mathcal{V}(\mathscr{H})$ (denoted w.r.t. $v$) if there exists a connected $\mathscr{H}$-graph that contains two vertices labeled $(\Lambda_0, v)$ and $(\Lambda_1, v)$, respectively. 
	\end{definition}
	
	\begin{lemma}\label{eqrel}
		Suppose that there exists an $\mathscr{H}$-path $E_1, ..., E_r$ that connects two vertices labeled $(\Lambda_0, v)$ and $(\Lambda_1, w)$, respectively. Then, for every edges $e,f \in \mathcal{E}(\mathscr{H})$ satisfying $\src(e) = v$ and $\trg(f) = w$, there exists an $\mathscr{H}$-path of type $e,...,e_1, ..., e_r, ..., f$ that connects two vertices labeled $(\Lambda_0, v)$ and $(\Lambda_1, w)$ and that contains $E_1, ..., E_r$ as a subpath.
	\end{lemma}
	
	\begin{proof}
		Let $E_1, ..., E_r$ be an $\mathscr{H}$-path labeled $e_1, ..., e_r$ (with $e_i \in \mathcal{E}(\mathscr{H})$ for every $i \in \llbracket 1, r \rrbracket$) such that \begin{itemize}
			\item $\src(E_1)$ is labeled $(\Lambda_0, v)$;
			\item $\trg(E_r)$ is labeled $(\Lambda_1, w)$.
		\end{itemize}
		Without loss of generality, one can assume this $\mathscr{H}$-path to be reduced.
		We first prove that there exists a reduced $\mathscr{H}$-path labeled $e,...,e_1, ..., e_r$ that connects a vertex labeled $(\Lambda_0, v)$ to a vertex labeled $(\Lambda_1, w)$, and that contains $E_1,...,E_r$ as a subpath. If $e = e_1$, then the edge path $E_1, ..., E_r$ is suitable. 
		
		Otherwise, let us denote by $(\rA, \B) = (\M_{e, \src}, \M_{e, \trg})$ and $v' = \trg(e)$.  Let us define \[\Lambda_0' = \B(\rA^{-1}\Lambda_0 \cap \mathbb{Z}^d).\]
		By construction, $(\Lambda_0, v)$ and $(\Lambda_0', v')$ satisfy the Transfer Equation \ref{transfer} \[\rA^{-1}\Lambda_0 \cap \mathbb{Z}^d = \B^{-1}\Lambda_0' \cap \mathbb{Z}^d.\]
		Hence there exists an $\mathscr{H}$-graph which consists of a single edge $E_0'$ labeled $e$ connecting a vertex $V_0$ labeled $(\Lambda_0, v)$ to a vertex $V_1$ labeled $(\Lambda_0', v')$.
		\paragraph{Case 1: Let us first assume that $|\det(\B)| \geq 2$.}
		As $|\det(B)| \geq 2$ and $\Lambda_0' \subseteq \B\mathbb{Z}^d$, one has \[\left|\mathbb{Z}^d/\langle \B\mathbb{Z}^d , \Lambda_0' \rangle\right| = |\det(\B)| \geq 2.\]
		Hence the labeled graph which consists of \begin{itemize}
			\item the edge $E_0'$; 
			\item an edge $E_1'$ labeled $\overline{e}$ connecting $V_1$ to a vertex $V_2$ labeled $(\Lambda_0, v)$
		\end{itemize}
		is an $\mathscr{H}$-graph of type $e, \overline{e}$ that connects a vertex labeled $(\Lambda_0, v)$ to a vertex labeled $(\Lambda_0, v)$. 
		Finally, as $e \neq e_1$ by assumption the labeled graph obtained by the concatenation of the $\mathscr{H}$-paths $E_0', E_1'$ and $E_1, ..., E_r$ is an $\mathscr{H}$-graph of type $e, \overline{e}, e_1, ..., e_r$ that connects a vertex labeled $(\Lambda_0, v)$ to a vertex labeled $(\Lambda_1, w)$ as required. 
		\paragraph{Case 2: Otherwise, $|\det(\B)| =1$.} In particular, the graph $\mathscr{H}$ being reduced, $e$ is a loop. We distinguish two subcases:
		\subparagraph{Subcase 2.1:} There exists an edge $f \neq \overline{e}$ with source $v$ such that $|\det(\M_{f, \trg})| \geq 2$. Hence we can apply Case 1 to obtain an $\mathscr{H}$-path $E_1', E_2'$ of type $f, \overline{f}$ that connects a vertex labeled $(\Lambda_0', v)$ to a vertex labeled $(\Lambda_0', v)$. As $f \neq \overline{e}$ and $(\Lambda_0, v)$ and $(\Lambda_0', v')$ satisfy the Transfer Equation \ref{transfer}, the labeled graph which consists of \begin{itemize}
			\item the concatenation of the $\mathscr{H}$-path $E_0'$ with the $\mathscr{H}$-path $E_1', E_2'$;
			\item an edge $E_3'$ labeled $\overline{e}$ with source $\trg(E_2')$ and target a new vertex labeled $(\Lambda_0,v)$
		\end{itemize}
		is an $\mathscr{H}$-graph. As $e \neq e_1$, we can concatenate the $\mathscr{H}$-path $E_0',E_1',E_2',E_3'$ and the $\mathscr{H}$-path $E_1, ..., E_r$ to obtain an $\mathscr{H}$-path labeled $e, f, \overline{f}, \overline{e}, e_1, ..., e_r$ that connects a vertex labeled $(\Lambda_0,v)$ to a vertex labeled $(\Lambda_1, w)$. 
		\subparagraph{Subcase 2.2:} Otherwise, as $G$ is non-amenable, Proposition~\ref{amenable} implies that $\mathscr{H}$ consists of a collection of at least two loops based at a single vertex such that every label (except possibly $\rA$) is in $\GL_d(\mathbb{Z})$. As $G$ is not a semidirect product $\mathbb{Z}^d \rtimes \mathbb{F}_r$, one has necessarily $|\det(A)| \geq 2$. Let $f \in \mathscr{H} \setminus \{e, \overline{e}\}$.
		Subcase 2.1 delivers an $\mathscr{H}$-path $E_1',E_2',E_3',E_4'$ of type $f, \overline{e}, e, \overline{f}$ that connects a vertex labeled $(\Lambda_0', v)$ to a vertex labeled $(\Lambda_0', v)$. As $f \neq \overline{e}$, the $\mathscr{H}$-graph which consists of \begin{itemize}
			\item the concatenation of the $\mathscr{H}$-path $E_0'$ with the $\mathscr{H}$-graph $E_1', E_2', E_3', E_4'$;
			\item an edge $E_5'$ of type $\overline{e}$ that connects $\trg(E_4'$) to a new vertex labeled $(\Lambda_0, v)$
		\end{itemize}
		is an $\mathscr{H}$-path. As $e \neq e_1$, we can concatenate the $\mathscr{H}$-path $E_0',E_1',E_2',E_3', E_4', E_5'$ and the $\mathscr{H}$-path $E_1, ..., E_r$ to obtain an $\mathscr{H}$-path labeled $e, f, \overline{e}, e, \overline{f}, \overline{e}, e_1, ..., e_r$ that connects a vertex labeled $(\Lambda_0,v)$ to a vertex labeled $(\Lambda_1, w)$. 
		
		Hence we proved that there exists an $\mathscr{H}$-path labeled $e, ..., e_1, ..., e_r$ that connects a vertex labeled $(\Lambda_0, v)$ to a vertex labeled $(\Lambda_1, w)$. Using this result on the reverse path (that connects a vertex labeled $(\Lambda_1, w)$ to a vertex labeled $(\Lambda_0, v)$) leads to an $\mathscr{H}$-path of type $e, ..., e_1, ..., e_r, ..., f$ that connects a vertex labeled $(\Lambda_0, v)$ to a vertex labeled $(\Lambda_1, w)$, which leads to the conclusion.
	\end{proof}
	
	\begin{corollary}\label{coreq}
		$\mathscr{H}$-equivalence with respect to a prescribed vertex $v$ is an equivalence relation. 
		
		More precisely, the following holds: if $E_1, ..., E_r$ (\textit{resp.} $F_1, ..., F_s$) is an $\mathscr{H}$-path that connects a vertex labeled $(\Lambda_0, v)$ (\textit{resp.} $(\Lambda_1, v)$) to a vertex labeled $(\Lambda_1, v)$ (\textit{resp.} $(\Lambda_2, v)$), then there exists an $\mathscr{H}$-path that contains $E_1, ..., E_{r-1}$ and $F_2,..., F_s$ as disjoint (labeled) subpaths.
	\end{corollary}
	
	\begin{proof}
		Let $\Lambda_0, \Lambda_1, \Lambda_2 \leq \mathbb{Z}^d$.
		
		As the empty path is an $\mathscr{H}$-path that connects any vertex labeled $(\Lambda_0, v)$ to itself, $\Lambda_0$ is $\mathscr{H}$-equivalent to itself (w.r.t. $v$).
		
		If $\Lambda_0$ is $\mathscr{H}$-equivalent to $\Lambda_1$ w.r.t. $v$, then there exists a reduced $\mathscr{H}$-path $E_1, ..., E_r$ with source labeled $(\Lambda_0, v)$ and target labeled $(\Lambda_1, v)$. The reversed $\mathscr{H}$-path $\overline{E_r}, ..., \overline{E_1}$ has its source labeled $(\Lambda_1, v)$ and its target labeled $(\Lambda_0, v)$. Thus, $\Lambda_1$ is $\mathscr{H}$-equivalent to $\Lambda_0$ w.r.t. $v$.
		
		Let us assume that $\Lambda_0$ is $\mathscr{H}$-equivalent to $\Lambda_1$ w.r.t. $v$ and that $\Lambda_1$ is $\mathscr{H}$-equivalent to $\Lambda_2$ w.r.t. $v$. Let $E_1, ..., E_r$ be a reduced $\mathscr{H}$-path of type $e_1, ..., e_r$ with source labeled $(\Lambda_0, v)$ and target labeled $(\Lambda_1, v)$ and let $F_1, ..., F_s$ be a reduced $\mathscr{H}$-path of type $f_1, ..., f_s$ with source labeled $(\Lambda_1, v)$ and target labeled $(\Lambda_2, v)$. 
		\paragraph{Case 1: Let us assume that $f_1 \neq \overline{e_r}$.} Then, the concatenation of these $\mathscr{H}$-paths delivers an $\mathscr{H}$-path with source labeled $(\Lambda_0, v)$ and target labeled $(\Lambda_2, v)$. 
		\paragraph{Case 2: Otherwise, let us denote by $e = e_r = \overline{f_1}$.} 
		\subparagraph{Subcase 2.1: First, we assume that there exists an edge $g \neq e_r$ such that $\trg(g) = v$.} By Lemma \ref{eqrel}, there exists an $\mathscr{H}$-path $E_1', ..., E_t'$ of type $e_1, ..., e_r, ..., g$ that contains $E_1, ..., E_r$ and that connects a vertex labeled $(\Lambda_0, v)$ to a vertex labeled $(\Lambda_1, v)$. As $g \neq \overline{f_1}$, the concatenation of the $\mathscr{H}$-paths $E_1', ..., E_t'$ and $F_1, ..., F_s$ delivers an $\mathscr{H}$-path with source labeled $(\Lambda_0, v)$ and target labeled $(\Lambda_2, v)$. 
		\subparagraph{Subcase 2.2: Now we assume that $e$ is the unique edge with target $v$.} Notice that in this case, $e$ cannot be a loop. If $r=0$ or $s=0$, we are done. Hence we assume that $r, s \geq 1$. Let $u = \src(e)$ and $(\rA, \B) := (\M_{e, \src}, \M_{e, \trg})$. Let $\left(\widetilde{\Lambda_1}, u\right)$ and $\left(\widetilde{\Lambda_2}, u\right)$ be the labels of $\src(E_r)$ and $\trg(F_1)$, respectively. By the Transfer Equation \ref{transfer} we have \begin{align*}
			\rA^{-1}\widetilde{\Lambda_1} \cap \mathbb{Z}^d &= \B^{-1}\Lambda_1 \cap \mathbb{Z}^d \\
			&= \rA^{-1}\widetilde{\Lambda_2} \cap \mathbb{Z}^d.
		\end{align*}
		Let us define \[\Lambda_1' = \B\left(\rA^{-1}\widetilde{\Lambda_1} \cap \mathbb{Z}^d\right).\]
		As $e$ is not a loop and $\mathscr{H}$ is reduced, one has $|\det(B)| \geq 2$. Hence the labeled graph which consists of \begin{itemize}
			\item an edge $E_0'$ labeled $e$ with source labeled $\left(\widetilde{\Lambda_1}, u\right)$ and target $V$ labeled $(\Lambda_1', v)$;
			\item an edge $E_1'$ labeled $\overline{e}$ with source $V$ and target a vertex labeled $\left(\widetilde{\Lambda_2}, u\right)$
		\end{itemize}
		is an $\mathscr{H}$-graph. 
		Hence the concatenation of the $\mathscr{H}$-path $E_1, ..., E_{r-1}$, the $\mathscr{H}$-path $E_0', E_1'$ and of the $\mathscr{H}$-path $F_2, ..., F_s$ delivers an $\mathscr{H}$-path that connects a vertex labeled $(\Lambda_0, v)$ to a vertex labeled $(\Lambda_2, v)$. 
		
		In any case we proved that $\Lambda_0$ is $\mathscr{H}$-equivalent to $\Lambda_2$ w.r.t. $v$. 
	\end{proof}
	
	\begin{lemma}\label{cycle}
		Let $\Lambda_0$ be a subgroup of $\mathbb{Z}^d$ and $v \in \mathcal{V}(\mathscr{H})$. There exists a finite $\mathscr{H}$-graph \begin{itemize}
			\item that is not simply connected;
			\item that contains a vertex labeled $(\Lambda_0, v)$;
			\item that contains at least two non-saturated vertices.
		\end{itemize} 
	\end{lemma}
	
	\begin{proof}
		We distinguish two cases:
		\paragraph{Case 1: Assume that there exists a loop $e_0 \in \mathcal{E}(\mathscr{H})$ based at $v$.} Let us consider an $\mathscr{H}$-path $E_1, E_2$ labeled $e_0, e_0$ such that \begin{itemize} 
        \item the vertex $\trg(E_1) = \src(E_2)$ is labeled $(\Lambda_0, v)$;
        \item the vertex $\src(E_1)$ is labeled $(\Lambda_1,v)$, where $\Lambda_1 = \M_{e_0, \src}\left(\M_{e_0, \trg}^{-1}\Lambda_0 \cap \mathbb{Z}^d\right)$ (so that $\Lambda_1$, $\Lambda_0$ and $e_0$ satisfy the Transfer Equation \eqref{transfer});
         \item the vertex $\trg(E_2)$ is labeled $(\Lambda_2,v)$, where $\Lambda_2 = \M_{e_0, \trg}\left(\M_{e_0, \src}^{-1}\Lambda_0 \cap \mathbb{Z}^d\right)$ (so that $\Lambda_0$, $\Lambda_2$ and $e_0$ satisfy the Transfer Equation \eqref{transfer}).\end{itemize} 
        We apply Lemma \ref{eqrel} to get an $\mathscr{H}$-path labeled $e_0, ..., e_0$ with source $\trg(E_2)$ and target $\src(E_1)$. 
        \subparagraph{Subcase 1.a:} If there exists $f_0 \in \mathcal{E}(\mathscr{H}) \setminus \{e_0, \overline{e_0}\}$, then the vertices $\trg(E_2)$ and $\src(E_1)$ are neither saturated relatively to $f_0$ nor to $\overline{f_0}$, which leads to the conclusion in this particular case. 
        \subparagraph{Subcase 1.b:} Otherwise, as $G$ is non-amenable, one has $|\det(\M_{e_0, \src})|\geq 2$ and $|\det(\M_{e_0, \trg})|\geq 2$ by Proposition \ref{amenable}. In particular: \begin{itemize}
            \item as $\src(E_1)$ is labeled by a subgroup $\Lambda_1$ of $\M_{e_0, \src}\mathbb{Z}^d$, one has \begin{align*}\left|\mathbb{Z}^d/\left\langle \Lambda_1, \M_{e_0, \src}\mathbb{Z}^d\right\rangle \right|&= \left|\det(\M_{e_0, \src})\right| \\&\geq 2,\end{align*} so, as $\src(E_1)$ has a single outgoing edge labeled $e_0$, the vertex $\src(E_1)$ is non-saturated relatively to $e_0$;
            \item likewise, the vertex $\trg(E_2)$ is non-saturated relatively to $\overline{e_0}$.
        \end{itemize}
		\paragraph{Case 2: Otherwise, we fix an edge $e_0 \in \mathcal{E}(\mathscr{H})$ with source $v$ such that $w :=\trg(e_0) \neq v$.} In particular, the graph of groups $\mathscr{H}$ being reduced, denoting by $(\rA, \B) := \left(\M_{e_0, \src}, \M_{e_0, \trg}\right)$, one has $|\det(A)| \geq 2$ and $|\det(B)| \geq 2$. Let us define an $\mathscr{H}$-path $E_1, E_2, E_3, E_4, E_5,E_6,E_7,E_8$ as follows: 
		\begin{itemize}
			\item $E_1$, $E_3$, $E_5$, $E_7$ are labeled $e_0$ and $E_2$, $E_4$, $E_6$, $E_8$ are labeled $\overline{e_0}$;
			\item $\src(E_1)$ is labeled $(\Lambda_0, v)$;
			\item the vertices $\trg(E_1)$, $\trg(E_3)$, $\trg(E_5)$ and $\trg(E_7)$ are all labeled $(\B(\rA^{-1}\Lambda_0 \cap \mathbb{Z}^d), v)$;
			\item the vertices $\trg(E_2)$, $\trg(E_4)$, $\trg(E_6)$ and $\trg(E_8)$ are labeled $(\Lambda_0 \cap \rA\mathbb{Z}^d, v)$.
		\end{itemize}
        Notice that at least four vertices are non-saturated relatively to some edge $\widehat{e} \in \mathcal{E}(\mathscr{H})$. Indeed, by Proposition \ref{amenable}, as $G$ is non-amenable and $\mathscr{H}$ is reduced: \begin{itemize}
		    \item either $|\det(\rA)|\geq 3$. In this case, the vertices $\trg(E_2)$, $\trg(E_4)$, $\trg(E_6)$ and $\trg(E_8)$ are non-saturated relatively to $e_0$;
            \item or $|\det(\B)| \geq 3$. In this case, $\trg(E_1)$, $\trg(E_3)$, $\trg(E_5)$ and $\trg(E_7)$ are non-saturated relatively to $\overline{e_0}$;
            \item or there exists an edge $f_0 \neq e_0$ with source $v$. In this case, the vertices $\trg(E_2)$, $\trg(E_4)$, $\trg(E_6)$ and $\trg(E_8)$ are non-saturated relatively to $f_0$;
            \item or there exists an edge $f_0 \neq \overline{e_0}$ with target $w$. In this case, $\trg(E_1)$, $\trg(E_3)$, $\trg(E_5)$ and $\trg(E_7)$ are non-saturated relatively to $f_0$.
		\end{itemize}
        In any case, denoting by $V_1,V_2,V_3,V_4$ four non-saturated vertices relatively to some edge $\widehat{e}$, Lemma \ref{eqrel} delivers a reduced $\mathscr{H}$-path $F_1, ..., F_r$ of type $\widehat{e}, ..., \overline{\widehat{e}}$ with source and target labeled by the label of $V_1$ and $V_2$, respectively. Gluing its source to $V_1$ and its target to $V_2$ leads to an $\mathscr{H}$-graph $\mathcal{C}$ in which the vertices $V_3$ and $V_4$ are non-saturated relatively to $\widehat{e}$. Hence $\mathcal{C}$ is suitable.
	\end{proof}
	
	\section{Perfect kernel of non-amenable GBS groups}\label{perfectkernel}
	
	The goal of this section is to give an explicit description of the perfect kernel in the case where the GBS group $G$ is non-amenable, \textit{i.e.} is defined neither by a single loop with at least one invertible matrix nor by a segment with two matrices having determinant $\pm 2$. 

	We start with the following lemma, that gives an inclusion in a more general setting: 
	
	\begin{lemma}\label{inclusiond}
		Let $G$ be the fundamental group of a finite graph of groups $\mathscr{H}$ such that, for every $v \in \mathcal{V}(\mathscr{H})$, the group $G_v$ is Noetherian (\textit{i.e.} every subgroup of $G_v$ is finitely generated (\textit{e.g.} the groups $G_v$ are finitely generated abelian groups)). Let us denote by $\mathcal{T}$ the Bass-Serre tree of $\mathscr{H}.$ Then \[\mathcal{K}(G) \subseteq \{H \leq G \mid H \backslash \mathcal{T} \text{\ is infinite}\}. \]  
	\end{lemma}
	
	\begin{proof}
		Let $H$ be a subgroup of $G$ whose graph of groups $H \backslash \mathcal{T}$ is finite.
		Let us show that under the assumptions of the lemma, the set \[\Omega = \{H' \leq G, H \leq H'\}\] is a countable neighborhood of $H$ in $\Sub(G)$. 
		
		First notice that every element of $\Omega$ has a finite graph of groups. In particular, any element of $\Omega$ is finitely generated: denoting by $\mathscr{K}$ its graph of groups, Bass-Serre theory tells us that it is generated by a finite number of subgroups of the groups $G_v$ (one per each vertex of $\mathscr{K}$, each of these being finitely generated by noetherianity of the groups $G_v$), and by one element per edge of $\mathscr{K}$. In particular: \begin{itemize}
			\item as $H$ belongs to $\Omega$, it is finitely generated so $\Omega$ is an open neighborhood of $H$;
			\item $\Omega$ is included in the subset of finitely generated subgroups, hence is countable.
		\end{itemize}
		This proves that $H$ has a countable neighborhood, thus $H \notin \mathcal{K}(G)$.
	\end{proof}

Now we prove Theorem \ref{computation}:
	
	\begin{theorem}\label{kerneld}
		Let $G$ be a non-amenable GBS group defined by a reduced graph of groups $\mathscr{H}$ and let $\mathcal{T}$ be the associated Bass-Serre tree. Then \[\mathcal{K}(G) = \{H \leq G \mid H \backslash \mathcal{T} \text{ \ is infinite} \}.\]
	\end{theorem}
	
	\begin{proof}
		The inclusion $\mathcal{K}(G) \subseteq \{H \leq G \mid H \backslash \mathcal{T} \text{ \ is infinite} \}$ follows from Lemma~\ref{inclusiond}. By the correspondence between $\mathscr{H}$-graphs and quotients of $\mathcal{T}$ by subgroups of $G$ established in Remark~\ref{eqbsh}, it suffices to show that any subgroup $H$ of $G$ whose $\mathscr{H}$-graph $\mathcal{G}$ is infinite belongs to $\mathcal{K}(G)$. Let us fix $v \in \mathcal{V}(\mathscr{H})$. Let us denote by $\alpha$ the associated $\mathscr{H}$-preaction. Let $\beta$ be a sub-$\mathscr{H}$-preaction of $\alpha$ whose $\mathscr{H}$-graph is a finite subgraph $K$ of $\mathcal{G}$. By assumption, $K$ has a vertex $V_0$ labeled $(\Lambda_0, v_0)$ for some $\Lambda_0 \leq \mathbb{Z}^d$ and some $v_0 \in \mathcal{V}\left(\mathscr{H}\right)$, which is not saturated relatively to some edge $e_0$ with source $v_0$. By Lemma \ref{forest}, there exists a saturated $\mathscr{H}$-graph $\mathcal{G}_0$ that contains $K$ and such that the quotient $\mathcal{G}_0/K$ is an infinite forest. Hence, by Lemma \ref{completiond}, there exists a saturated $\mathscr{H}$-preaction $\alpha_0$ that extends $\beta$ and whose $\mathscr{H}$-graph is $\mathcal{G}_0$.
		
		Let us first assume that $\mathscr{H}$ does not consist of a single vertex and a collection of loops labeled by matrices in $\GL_d(\mathbb{Z})$. By Lemma \ref{cycle}, there exists a non-simply connected finite $\mathscr{H}$-graph $\mathcal{C}$ that contains \begin{itemize}
			\item a vertex labeled $(\Lambda_0, v)$;
			\item two vertices $V$ and $W$ labeled $(\Lambda_1, v_1)$ and $(\Lambda_2, v_2)$ that are non-saturated relatively to some edges denoted by $e, f$, respectively.
		\end{itemize}
		By Lemma \ref{fini}, $\mathcal{C}$ is the $\mathscr{H}$-graph of an $\mathscr{H}$-preaction $\gamma$. Lemma \ref{eqrel} delivers an $\mathscr{H}$-path $\mathcal{P}$ of type $e_0,...,\overline{e}$ that connects $V_0$ to $V$. The vertex $W$ is non-saturated relatively to $f$ in the $\mathscr{H}$-graph $K' := K \cup \mathcal{P} \cup \mathcal{C}$. By Lemma \ref{forest}, there exists a saturated $\mathscr{H}$-graph $\mathcal{G}_1$ that contains $K'$ and such that the quotient $\mathcal{G}_1/K'$ is a forest. Hence, by Lemma \ref{completiond}, there exists a saturated $\mathscr{H}$-preaction $\alpha_1$ that extends both $\beta$ and $\gamma$ and whose $\mathscr{H}$-graph is $\mathcal{G}_1$. As $\mathcal{G}_0$ and $\mathcal{G}_1$ are non-isomorphic (because they do not share the same homotopy type), the (saturated) $\mathscr{H}$-preactions $\alpha_0$ and $\alpha_1$ both extend $\beta$ and their stabilizers in $G$ are not conjugate. We proved that any neighborhood of $\alpha$ (or equivalently, of $H$) contains two non-isomorphic $G$-actions $\alpha_0$, $\alpha_1$ (equivalently, two non-conjugate subgroups $H_0$, $H_1$ of $G$) whose $\mathscr{H}$-graphs $\mathcal{G}_0$, $\mathcal{G}_1$ are infinite (\textit{i.e.}, the quotient graphs $H_0 \backslash \mathcal{T}$, $H_1 \backslash \mathcal{T}$ are infinite). This proves that the space $\{H \leq G \mid H \backslash \mathcal{T} \text{ \ is infinite} \}$ has no isolated point. Thus, this space is included in $\mathcal{K}(G)$.

		Otherwise, our group $G$ is of the form $\mathbb{Z}^d \rtimes \mathbb{F}_r$ (where $r \geq 2$ denotes the number of loops) and each generator of $\mathbb{F}_r$ acts on $\mathbb{Z}^d$ by multiplication by an invertible integer matrix.
		Denoting by $\pi : G = \mathbb{Z}^d \rtimes \mathbb{F}_r \to \mathbb{F}_r$ the canonical surjection, every subgroup $H$ of $G$ is fully determined by its intersection $H_0 = H \cap \mathbb{Z}^d$ with $\mathbb{Z}^d$, its image $\pi(H) \leq \mathbb{F}_r$ under $\pi$ satisfying \[x \cdot H_0 = H_0 \ \forall x \in \pi(H)\]
		and, given a basis $(a_i)_{i \in I}$ of $\pi(H)$ (the set $I$ being countable), elements $(u_i, a_i) \in H$ for every $i \in I$. Notice that in this case, the Bass-Serre tree of $G$ is the Cayley graph of $\mathbb{F}_r$ with respect to the standard generating set, and that for any subgroup $H$ of $G$, the quotient graph $H \backslash \mathcal{T}$ is infinite if and only if $\pi(H) \in \Sub_{[\infty]}(\mathbb{F}_r)$. One distinguishes two cases: 
		
		\begin{enumerate}
			
			\item Let us first assume that $\rk(H_0) = d$. Let us denote by $(\rA, \B) = (\M_{e_0, \src}, \M_{e_0, \trg})$, let us consider an edge $f_0 \neq e_0$ and let us write $(\C, \D) = (\M_{f_0, \src}, \M_{f_0, \trg})$. As the subgroup of $\GL_d(\mathbb{Z})$ generated by $\D\C^{-1}$  acts on the (finite) set of lattices of determinant $\pm \det(H_0)$, there exists an integer $k \in \mathbb{N}^*$ satisfying \[\left(\D\C^{-1}\right)^kH_0 = H_0.\]
			In particular, there exists an $\mathscr{H}$-cycle $\mathcal{C} = E_1,...,E_k$ all of whose edges are of type $f_0$ such that $\src(E_i)$ is labeled $\left(\left(\D\C^{-1}\right)^iH_0, \src(e_0)\right)$ for every $i \in \llbracket 1, k \rrbracket$. As no vertex of $\mathcal{C}$ is saturated relatively to $e_0$, Lemma \ref{eqrel} delivers an $\mathscr{H}$-path $\mathcal{P}$ of type $e_0,...,\overline{e_0}$ that connects $V_0$ to a vertex $V$ of $\mathcal{C}$. Any other vertex of $\mathcal{C}$ is not saturated relatively to $e_0$ in $K' := K \cup \mathcal{P} \cup \mathcal{C}$. Hence we can conclude as in the previous case.         
			\item If $\rk(H_0) < d$, let us write $H_0$ in its Smith normal form: \[H_0 = \mathrm{P} \diag(d_1,\ldots,d_r,0,\ldots,0)\mathbb{Z}^d\] 
			(with $\mathrm{P} \in \GL_d(\mathbb{Z})$, $r < d$ and $d_i |d_{i+1}$ for every $i \in \llbracket 1,r-1 \rrbracket$). For every $N \in \mathbb{N}$, define \[H_0^{(N)} = \mathrm{P}\diag \left(d_1,\ldots,d_r, N\prod_{i=1}^rd_i,\ldots,N\prod_{i=1}^rd_i \right) \mathbb{Z}^d. \]
			
			\begin{claim}
			Every matrix stabilizing the subgroup $H_0$ also stabilizes $H_0^{(N)}$.
			\end{claim}
			\begin{cproof}
			Denoting by $\M$ a matrix stabilizing $H_0$, and by $e_1, \ldots, e_d$ the canonical basis of $\Z^d$, the matrix $\mathrm{P}^{-1}\M\mathrm{P}$ stabilizes $\langle d_1e_1, \ldots, d_re_r\rangle.$ In particular: \begin{align*}
			    &\mathrm{P}^{-1}\M\mathrm{P} \left\langle d_1e_1, \ldots, d_re_r, \left(N\prod_{i=1}^rd_i\right) e_{r+1}, \ldots, \left(N\prod_{i=1}^rd_i\right) e_{d} \right\rangle \\ &\leq \left\langle d_1e_1, \ldots, d_re_r , \left(N\prod_{i=1}^rd_i\right) \mathbb{Z}^d \right\rangle.
			\end{align*}
			Now, $\left(N\prod_{i=1}^rd_i\right) \mathbb{Z}^d \leq \left\langle d_1e_1, \ldots, d_re_r, \left(N\prod_{i=1}^rd_i\right) e_{r+1}, \ldots, \left(N\prod_{i=1}^rd_i\right) e_{d} \right\rangle$: if $j~\leq~r$, then $$\left(N\prod_{i=1}^rd_i \right)e_j = \left(N\prod_{i \neq j}d_i\right) (d_je_j) \in d_j e_j \mathbb{Z}.$$ Thus, $\mathrm{P}^{-1}\M\mathrm{P}$ stabilizes $\left\langle d_1e_1, \ldots, d_re_r, \left(N\prod_{i=1}^rd_i\right) e_{r+1}, \ldots, \left(N\prod_{i=1}^rd_i\right) e_{d} \right\rangle$, hence $\M$ stabilizes $H_0^{(N)}$. 
		    \end{cproof}
		    
			In particular, $H_0^{(N)}$ is $\pi(H)$-stable. Moreover, as $(a_i)_{i \in I}$ is a free basis of $\langle~a_i,~i~\in~I~\rangle$, one has \[\left( \mathbb{Z}^d \times 1 \right) \cap \langle (u_i, a_i), i \in I \rangle = \{1\}.\]
			
			This implies that \begin{align*}
				\widetilde{H}_N &:= \left\langle H_0^{(N)}, (u_i,a_i)_{i \in I} \right\rangle \\
				&= H_0^{(N)} \rtimes \left\langle \left(u_i,a_i\right)_{i \in I} \right\rangle.
			\end{align*} 
			As moreover $H_0^{(N)}$ tends to $H_0$ in $\Sub(\mathbb{Z}^d)$ as $N$ tends to $+\infty$, we deduce that the sequence of subgroups $	\widetilde{H}_N $ converges to $H$ (non-trivially, because $\widetilde{H}_N~\cap \mathbb{Z}^d~=~H_0^{(N)}$). As $\pi\left(\widetilde{H}_N\right) = \pi(H)$, the $\mathscr{H}$-graphs of $\widetilde{H}_N$ and $H$ have isomorphic skeletons (in particular, the subgroups $\widetilde{H}_N$ have infinite $\mathscr{H}$-graphs).
		\end{enumerate}
	\end{proof}
	
\begin{remark}\label{rknoncpct}
Notice that we did not make use of the cocompactness of the action $G \curvearrowright \mathcal{T}$. Thus, Theorem \ref{kerneld} extends to a larger class of groups that include GBS groups of rank $d$, \textit{i.e.} the class of non-amenable groups acting (non necessarily cocompactly) on an oriented tree with vertex and edge stabilizers isomorphic to $\mathbb{Z}^d$. In particular, in the case of a non-cocompact action on the Bass-Serre tree, the perfect kernel consists of the whole set of subgroups $\Sub(G)$. 
\end{remark}

The following corollary gives a class of GBS groups that satisfy the equality $\mathcal{K}(G) = \Sub_{[\infty]}(G)$. Recall that in rank $1$, the authors of \cite{solitar1} and \cite{bontemps} proved that this equality was true for non-unimodular GBS groups only. Let $v$ be a vertex of $\mathscr{H}$ and recall that the modular homomorphism $\Delta_G^{(v)}$ based at $v$ is defined by the following data: \begin{itemize}
    \item $\Delta_G^{(v)}$ is trivial on the vertex groups;
    \item for every edge generator $t_e$, denoting by $e_1,...,e_r$ the unique reduced edge path in $\mathscr{T}$ with source $v$ and target $\src(e)$, and by $e_{r+1},...,e_s$ the unique reduced edge path in $\mathscr{T}$ with source $\trg(e)$ and target $v$, one has \[\Delta_G^{(v)}(t_e) = \M_{e_1, \src}\M_{e_1, \trg}^{-1} \ldots \M_{e_r,\src}\M_{e_r, \trg}^{-1}  \M_{e, \src}\M_{e, \trg}^{-1} \M_{e_{r+1},\src}\M_{e_{r+1}, \trg}^{-1} \ldots \M_{e_s, \src}\M_{e_s, \trg}^{-1}.\]
\end{itemize}
	
	\begin{corollary}\label{cor}
		Let $G$ be a non-amenable and non-unimodular GBS group of rank $d$ such that the finite index subgroups of $\image\left(\Delta_{G}^{(v)}\right)$ are $\mathbb{Q}$-irreducible. Then \[\mathcal{K}(G) = \Sub_{[\infty]}(G).\]
	\end{corollary}
	
	Before proving Corollary \ref{cor}, we give an explicit example of a GBS group $G$ of rank $2$ whose perfect kernel consists of the set of infinite index subgroups of $G$. Let us define $\rA = \begin{pmatrix}
		2 & 2 \\
		2 & 4 \\
	\end{pmatrix}$ and $\B = \begin{pmatrix}
		1 & 0 \\
		0 & 2 \\
	\end{pmatrix}$ and let us define $G=\pi_1(\mathscr{H})$ as the fundamental group of the graph of groups defined in Figure \ref{figcor}. In other words, $G$ is defined by the following presentation: \[G \simeq \left\langle x, y, t \mid xy = yx, t^{-1}x^2y^2t = x, t^{-1}x^2y^4t = y^2 \right\rangle.\]
	
	\begin{figure}[ht]
		\center
		\begin{tikzpicture}
			\node[draw,circle,fill=gray!50] (a) at (0,0) {};
			\draw[>=latex, directed] (a) to [out=135,in=45,looseness=20] node[above]{$e$} node[very near start, below left]{$\rA$} node[very near end, below right]{$\B$} (a);
		\end{tikzpicture}
		\caption{The graph of groups $\mathscr{H}$.}
		\label{figcor}
	\end{figure}
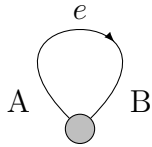
	The image of the modular homomorphism $\Delta_{G}$ is the subgroup of $\GL_2(\mathbb{Q})$ generated by $\rA\B^{-1} = \begin{pmatrix}
		2 & 1 \\
		2 & 2 \\
	\end{pmatrix}$. Hence $\image(\det \circ \Delta_G) = 2^{\mathbb{Z}}$, thus $G$ is not unimodular. Finite index subgroups of $\image(\Delta_G)$ are of the form $\left\langle \left(\rA\B^{-1}\right)^n \right\rangle$ (where $n \in \mathbb{Z} \smallsetminus \{0\}$). Notice that $\Spec_{\mathbb{R}}(\rA\B^{-1}) = \left\{2+\sqrt{2}, 2-\sqrt{2}\right\}$, so $\Spec_{\mathbb{R}}((\rA\B^{-1})^n) = \left\{\left(2+\sqrt{2}\right)^n, \left(2-\sqrt{2}\right)^n\right\}$ for every $n \in \mathbb{Z}$. Thus, for any $n \in \mathbb{Z} \setminus\{0\}$, one has $\Spec_{\mathbb{Q}}((\rA\B^{-1})^n) = \emptyset$, \textit{i.e.} the subgroup $\left\langle \left(\rA\B^{-1}\right)^n \right\rangle$ does not stabilize any non-trivial proper vector subspace of $\mathbb{Q}^2$. Hence, by Corollary \ref{cor}, one has $\mathcal{K}(G) = \Sub_{[\infty]}(G)$. 
	
	Now let us prove Corollary \ref{cor}. To begin with, we prove the following lemma:
	
	\begin{lemma}\label{commensurable}
		Let $H$ be a subgroup of $G$ and let us denote by $\mathcal{G}$ its $\mathscr{H}$-graph. Let $v \in \mathcal{V}(\mathscr{H})$ and let $(\Lambda_0, v)$ be the label of a vertex $V_0$ of $\mathcal{G}$. Then: \begin{align*}\left\{\Delta_{G}^{(v)}(g)\cdot \Lambda_0 \otimes \mathbb{Q} \mid g \in G\right\} &= \left\{\Lambda_0' \otimes \mathbb{Q} \mid \exists \text{a vertex $V \in \mathcal{V}(\mathcal{G})$ labeled $(\Lambda_0',v)$}\right\}  \\
			&= \left\{\Lambda_0'\otimes \mathbb{Q} \mid \text{$\Lambda_0'$ is $\mathscr{H}$-equivalent to $\Lambda_0$ w.r.t. $v$}.\right\} \end{align*}
	\end{lemma}
	
	\begin{proof}
		Let us prove that \[\left\{\Delta_{G}^{(v)}(g)\cdot \Lambda_0 \otimes \mathbb{Q} \mid g \in G\right\} \subseteq \left\{\Lambda_0'\otimes \mathbb{Q} \mid \exists \text{a vertex $V \in \mathcal{V}(\mathcal{G})$ labeled $(\Lambda_0', v)$}\right\}.\]
		Let $g \in G$. Let us consider a cycle $e_1, ..., e_r$ based at $v$ in ${\mathscr{H}}$ and elements $g_i \in G_{\src(e_i)}$ for every $i$ and $g_{r+1} \in G_v$ such that \[g = g_1 \cdot s_{e_1} \cdot g_2 \cdot ... \cdot s_{e_r} \cdot g_{r+1}\] 
		(where $s_{e_i} = t_{e_i}$ if $e_i \notin \mathscr{T}$ and $1$ otherwise). Denoting by $(\rA_i,\B_i)$ the label of $e_i$ for every $i \in \llbracket 1,r \rrbracket$, one has: \begin{equation}\label{delta}\Delta_{G}^{(v)}(g) = \rA_1 \B_1^{-1} \ldots \rA_r\B_r^{-1}.\end{equation}
		Let us consider an edge path $E_1, ..., E_r$ based at $V_0$ and labeled $\overline{e_r}, ..., \overline{e_1}$ in $\mathcal{G}$. For every $i \in \llbracket 1, r \rrbracket$, let $(\Lambda_i, \trg(e_i))$ be the label of $\trg(E_{i})$. By the Transfer Equation \ref{transfer}, we get \[\B_{r-i}^{-1}\Lambda_{i} \cap \mathbb{Z}^d = \rA_{r-i}^{-1} \Lambda_{i+1} \cap \mathbb{Z}^d\] for every $i \in \llbracket 0,r-1\rrbracket$,
		which implies that  \[\B_{r-i}^{-1}\Lambda_{i} \otimes \mathbb{Q} = \rA_{r-i}^{-1} \Lambda_{i+1} \otimes \mathbb{Q}.\]
		Thus \begin{equation}\Lambda_r \otimes \mathbb{Q} = \rA_1\B_1^{-1} \ldots \rA_r\B_r^{-1} \Lambda_0 \otimes \mathbb{Q}\end{equation}
		which implies, with Equation \ref{delta}, that $\Lambda_0' := \Lambda_r$ satisfies $\Lambda_0' \otimes \mathbb{Q} = \Delta_{G}^{(v)}(g) \cdot \Lambda_0 \otimes \mathbb{Q}$. 
		
		If there exists a vertex $V \in \mathcal{V}(\mathcal{G})$ labeled $(\Lambda_0', v)$, then $\Lambda_0'$ is ${\mathscr{H}}$-equivalent to $\Lambda_0$ w.r.t. $v$ by connectedness of $\mathcal{G}$. This proves that \begin{align*}&\left\{\Lambda_0'\otimes \mathbb{Q} \mid \exists \text{a vertex $V \in \mathcal{V}(\mathcal{G})$ labeled $(\Lambda_0', v)$}\right\} \\ &\subseteq \left\{\Lambda_0' \otimes \mathbb{Q} \mid \text{$\Lambda_0'$ is ${\mathscr{H}}$-equivalent to $\Lambda_0$ w.r.t. $v$}\right\}.\end{align*}

		Finally, let us prove that \[\left\{\Lambda_0'\otimes \mathbb{Q} \mid \text{$\Lambda_0'$ is ${\mathscr{H}}$-equivalent to $\Lambda_0$ w.r.t. $v$}\right\} \subseteq \left\{\Delta_{G}^{(v)}(g)\cdot \Lambda_0 \otimes \mathbb{Q} \mid g \in G\right\}.\]
		Let $\Lambda_0'$ be a subgroup which is ${\mathscr{H}}$-equivalent to $\Lambda_0$ w.r.t. $v$. Let us consider an ${\mathscr{H}}$-path $E_1, ..., E_r$ labeled $e_1, ..., e_r$ whose source is labeled $(\Lambda_0, v)$ and whose target is labeled $(\Lambda_0', v)$. For $i \in \llbracket 1, r \rrbracket$, let us denote by $(\Gamma_i, \src(e_i))$ the label of $\src(E_i)$ (hence $\Gamma_1 = \Lambda_0$). By the Transfer Equation \ref{transfer}, denoting by $(\rA_i, \B_i)$ the label of $e_i$ in ${\mathscr{H}}$ we get, as previously \[\Lambda_0' \otimes \mathbb{Q} = \B_r\rA_r^{-1} \ldots \B_1\rA_1^{-1} \Lambda_0 \otimes \mathbb{Q}.\] Thus, denoting by $g = s_{e_1}...s_{e_r}$ (where $s_{e_i} = t_{e_i}$ if $e_i \notin T$ and $1$ otherwise), one has: \[\Lambda_0' \otimes \mathbb{Q} = \Delta_{G}^{(v)}(g)^{-1} \cdot \Lambda_0 \otimes \mathbb{Q},\]
		which concludes the proof of the lemma.
	\end{proof}
	
	\begin{proof}[Proof of Corollary \ref{cor}]
		Let $H$ be a subgroup of $G$ whose graph of groups is finite. Let us show that $H$ has finite index under the assumptions of the corollary. Let $\mathcal{G}$ be the $\mathscr{H}$-graph of $H$. Let us denote by $(\Lambda_0,v)$ the label of a vertex of $\mathcal{G}$. By Lemma \ref{commensurable}, the orbit of $\Lambda_0 \otimes \mathbb{Q}$ under the action of $G$ \textit{via} $\Delta_{G}^{(v)}$ is finite. the stabilizer \[\{g \in G \mid \Delta_{G}^{(v)}(g) \cdot \Lambda_0 \otimes \mathbb{Q} = \Lambda_0 \otimes \mathbb{Q}\}\] has finite index in $G$. This implies that \begin{itemize}
			\item either $\Lambda_0 \otimes \mathbb{Q} = \mathbb{Q}^d$, \textit{i.e.} $\Lambda_0$ has finite index in $\mathbb{Z}^d$;
			\item or $\Lambda_0 = 0$.
		\end{itemize}
		Let us assume by contradiction that $\Lambda_0 = \{0\}$. By commensurability of the vertex stabilizers, this is equivalent to \[H \cap G_w = \{1\} \text{ \ for every vertex $w \in \mathcal{V}(\mathscr{H})$}\]
		(or equivalently, all the labels of the vertices of $\mathcal{G}$ are trivial). 
		
		As the image of the homomorphism $\left| \det \circ \Delta_{G}^{(v)} \right|$ is non-trivial, there exists a (reduced) cycle of edges $e_1,...,e_n$ in $\mathscr{H}$ based at $v$ and satisfying 
		\[\left|\frac{\prod_{i=1}^n\det(\rA_i)}{\prod_{i=1}^n\det(\B_i)}\right| \neq 1\]
		Let $n_i$ be the number of vertices of $\mathcal{G}$ labeled $\src(e_i)$. Every vertex labeled $\src(e_i)$ (\textit{resp.} $\trg(e_i)$) has $|\det(\rA_i)|$ (\textit{resp.} $|\det(\B_i)|$) outgoing (\textit{resp.} incoming) edges labeled $e_i$. Hence the number of edges labeled $e_i$ in $\mathcal{G}$ is \[n_i |\det(\rA_i)| = n_{i+1} |\det(\B_i)|.\]
		Combining these equalities for $i=1,...,n$, we get \[\prod_{i=1}^n|\det(\rA_i)| = \prod_{i=1}^n|\det(\B_i)|\]
		hence a contradiction. 
		Thus,  $\Lambda_0$ has finite index in $G_{\src(e_1)}$. By commensurability of the vertex stabilisers, this implies that $H \cap G_w$ has finite index in $G_w$ for every $w \in \mathcal{V}(\mathscr{H})$ which implies that $H \backslash G$ is finite (because $\mathcal{G}$ is also finite). 
		
		Using Theorem \ref{kerneld}, we finally deduce the equality \[\mathcal{K}(G) = \Sub_{[\infty]}(G).\]
	\end{proof}
	
	\begin{remark}
		If $G$ is a non-unimodular GBS group such that every non-trivial element of the image of the linear modular homomorphism $\Delta_{G}^{(v)}$ is $\mathbb{Q}$-irreducible, then the assumptions of Corollary \ref{cor} are satisfied. Let us explain why. As $G$ is non-unimodular, there exists $g \in G$ such that $\left|\det\left(\Delta_{G}^{(v)}(g)\right)\right| \neq 1$. In particular, $\Delta_{G}^{(v)}(g)$ has infinite order. Now assume by contradiction that there exists a non-trivial proper vector subspace $V$ of $\mathbb{Q}^d$ that is stabilized by a finite index subgroup $\Gamma$ of $\Delta_{G}^{(v)}(G)$. As $\Gamma$ has finite index, there exists some $n \in \mathbb{N}^*$ such that ${\Delta_G^{(v)}(g)}^n \in \Gamma$. Consequently, $V$ is stabilized by the non-trivial element $\Delta_G^{(v)}(g)^n$ of $\image\left(\Delta_G^{(v)}\right)$, which contradicts the assumption made on $G$.
	\end{remark}
	
	\section{A dynamical partition of the perfect kernel}\label{dynamicalpartition}
	
	The goal of this section is to extend the decomposition of the perfect kernel obtained in \cite{solitar1} and in \cite{bontemps} for non-amenable GBS groups of rank $1$. 
	
	\subsection{\texorpdfstring{Case where $G$ is not a semidirect product $\mathbb{Z}^d \rtimes \mathbb{F}_r$}{Case where $G$ is not a semidirect product $\mathbb{Z}^d \rtimes \mathbb{F}_r$}}
	
	In this section, we assume that $G$ is not a semidirect product $\mathbb{Z}^d \rtimes \mathbb{F}_r$, \textit{i.e.} $\mathscr{H}$ does not consist of a single vertex with a collection of loops all of whose labels are in $\GL_d(\mathbb{Z})$. This allows us to make use of the equivalence relation on $\Sub(\mathbb{Z}^d)$ defined in Section~\ref{eq}.
	
	Let us fix a vertex $v \in \mathcal{V}(\mathscr{H})$. We identify $G_v$ with $\mathbb{Z}^d$. Let us denote by $\simeq$ the $\mathscr{H}$-equivalence relation with respect to $v$ and by $\pi_v : \Sub(\mathbb{Z}^d) \to \Sub(\mathbb{Z}^d)/\simeq$ the canonical projection. Notice that the rank is constant on each fiber of $\pi_v$. We have the following statement: 
	
	\begin{proposition}\label{partitioninfinie}
		The set $\Sub(\mathbb{Z}^d)/\simeq$ is infinite countable. 
	\end{proposition}
	
	\begin{proof}
		As $\Sub(\mathbb{Z}^d)$ is countable, the set $\Sub(\mathbb{Z}^d)/\simeq$ is also countable. Let us show that it is infinite. 
		Let us define the finite subset $\mathcal{P}_{\mathscr{H}}$ of prime numbers \[\mathcal{P}_{\mathscr{H}} = \left\{p \in \mathcal{P} \mid \text{there exists an edge $e \in \mathcal{E}(\mathscr{H})$}, p \mid \det(\M_{e, \trg})\right\}.\]
		Let us define \[\delta : \begin{array}{ccccc}
			\mathcal{L}(\mathbb{Z}^d) & \to & \mathbb{Z} \\
			\Lambda & \mapsto & \prod_{p \notin \mathcal{P}_{\mathscr{H}}}p^{|\det(\Lambda)|_p} \\
		\end{array}.\]
		The image of $\delta$ is exactly the set of integers which are divisible by no element of $\mathcal{P}_{\mathscr{H}}$, hence is infinite. Let us show that $\delta$ is constant on the fibers of $\pi_v$, \textit{i.e.} that we have a factorization \[\xymatrix{
			\mathcal{L}(\mathbb{Z}^d) \ar[r]^{\delta} \ar[rd]^{\pi_v} & \delta(\mathbb{Z})  \\
			{} & \mathcal{L}(\mathbb{Z}^d)/\simeq \ar[u]^{\overline{\delta}} 
		}\]
		which will imply that $\mathcal{L}(\mathbb{Z}^d)/\simeq$ is infinite (hence $\Sub(\mathbb{Z}^d)/ \simeq$ is infinite).
		
		By a straightforward induction on the length of an $\mathscr{H}$-path, it suffices to prove that for any $\mathscr{H}$-edge labeled $e \in \mathcal{E}(\mathscr{H})$ with source labeled $(\Lambda_0, \src(e))$ and target labeled $(\Lambda_1, \trg(e))$, one has $\delta(\Lambda_0) = \delta(\Lambda_1)$. Let $E$ be such an edge. By the Transfer Equation \ref{transfer}, denoting by $(\rA, \B) = (\M_{e, \src}, \M_{e, \trg})$, one has \[\rA^{-1}\Lambda_0 \cap \mathbb{Z}^d = \B^{-1} \Lambda_1 \cap \mathbb{Z}^d.\]
		Let $p \in \mathcal{P} \setminus \mathcal{P}_{\mathscr{H}}.$ We have \[\frac{\det(\Lambda_0 \cap \rA\mathbb{Z}^d)}{\det(A)} = \frac{\det(\Lambda_1 \cap \B\mathbb{Z}^d)}{\det(B)},\]
		hence, as $p \nmid \det(\rA)$ and $p \nmid \det(\B)$: \begin{equation}\label{eqdet}|\det(\Lambda_0 \cap \rA\mathbb{Z}^d)|_p = |\det(\Lambda_1 \cap \B\mathbb{Z}^d)|_p.\end{equation}
		From $\det(\rA)\Lambda_0 \leq \Lambda_0 \cap \rA\mathbb{Z}^d \leq \Lambda_0$, we get \[\det(\Lambda_0) \mid \det(\Lambda_0 \cap \rA\mathbb{Z}^d) \mid \det(\Lambda_0)(\det(\rA))^d.\]
		Thus, as $p \nmid \det(\rA)$: \[|\det(\Lambda_0)|_p = |\det(\Lambda_0 \cap \rA\mathbb{Z}^d)|_p.\]
		Likewise \[|\det(\Lambda_1)|_p = |\det(\Lambda_1 \cap \B\mathbb{Z}^d)|_p,\]
		so by Equation \ref{eqdet} we finally get \[|\det(\Lambda_0)|_p = |\det(\Lambda_1)|_p.\]
		As this is true for any $p \notin \mathcal{P}_{\mathscr{H}}$, this implies that \[\delta(\Lambda_0) = \delta(\Lambda_1)\] as required.
	\end{proof}
	
	Let us define the \textbf{$\mathscr{H}$-phenotype with respect to $v$} as the following function: \[\Ph_{\mathscr{H},v} : \begin{array}{ccccc}
		\Sub(G) & \to & \Sub(\mathbb{Z}^d)/\simeq \\
		H & \mapsto & \pi_v(H \cap G_v)  \\
	\end{array}.\]
	
	\begin{proposition}
		The $\mathscr{H}$-phenotype $\Ph_{\mathscr{H},v}$ is surjective and invariant under conjugation by any element of $G$.
	\end{proposition}
	
	\begin{proof}
		The surjectivity of $\Ph_{\mathscr{H}, v}$ results from the surjectivity of $\pi_v$ and of the function $\begin{array}{ccccc}
			\Sub(G) & \to & \Sub(G_v) \\
			H & \mapsto & H \cap G_v  \\
		\end{array}.$
		
		Let $H$ be a subgroup of $G$ and $g \in G$. Let $\Lambda_0 := H \cap G_v \leq \mathbb{Z}^d$ and $\Lambda_1 := gHg^{-1} \cap G_v \leq \mathbb{Z}^d$. Let $\mathcal{G}$ be the $\mathscr{H}$-graph of $H$. By definition, there exist two vertices labeled $(\Lambda_0, v)$ and $(\Lambda_1, v)$ in $\mathcal{G}$. Hence, by connectedness of $\mathcal{G}$, one has $\pi_v(\Lambda_0) = \pi_v(\Lambda_1)$.
	\end{proof}

    As the group $G_v$ acts non-cocompactly on the Bass-Serre tree $\mathcal{T}$, any subgroup of $G_v$ lies in the perfect kernel $\mathcal{K}(G)$ by Theorem \ref{kerneld}. In particular, the restriction $\begin{array}{ccccc}
			\mathcal{K}(G) & \to & \Sub(G_v)/\simeq \\
			H & \mapsto & \pi_v(H \cap G_v)  \\
		\end{array}$ remains surjective and invariant under conjugation.

    This function leads to a dynamical partition of the perfect kernel \begin{equation}\label{dynpartperfker}\mathcal{K}(G) = \bigsqcup_{\Lambda \leq \mathbb{Z}^d}\mathcal{K}(G) \cap\Ph_{\mathscr{H},v}^{-1}(\pi_v(\Lambda)).\end{equation} By the previous remark and Proposition \ref{partitioninfinie}, this partition is infinite countable. Now we are able to prove Theorem \ref{decomposition} which gives a description of the dynamics induced on each piece of the aforementioned decomposition: 
	
	\begin{theorem}\label{dynamic}
		For any $\Lambda_0 \leq \mathbb{Z}^d$, there exists a dense orbit in $\Ph_{\mathscr{H},v}^{-1}(\pi_v(\Lambda_0)) \cap \mathcal{K}(G)$.
	\end{theorem}
	
	\begin{proof}
		Let $\Lambda \leq \mathbb{Z}^d$ and let $(H_i)_{i \in  \mathbb{N}} \in \left(\mathcal{K}(G) \cap \Ph_{\mathscr{H},v}^{-1}(\pi_v(\Lambda))\right)^{\mathbb{N}}$ be the sequence of finitely generated subgroups lying in $\mathcal{K}(G) \cap \Ph_{\mathscr{H},v}^{-1}(\pi_v(\Lambda))$. For $i \in \mathbb{N}$, let $\alpha_i$ be the saturated $\mathscr{H}$-preaction associated to $H_i$ and let $\mathcal{G}_i$ be the $\mathscr{H}$-graph of $\alpha_i$. Let $\beta_i$ be a sub-$\mathscr{H}$-preaction of $\alpha_i$ whose stabilizer is $H_i$ and whose $\mathscr{H}$-graph $K_i$ is finite (legit, because $H_i$ is finitely generated). As $H_i \in \mathcal{K}(G)$, Theorem \ref{kerneld} implies that there exists a vertex $V_i \in \mathcal{V}(K_i)$ labeled $(\Lambda_i, v_i)$ which is non-saturated relatively to some edge $e_i \in \mathcal{E}(\mathscr{H})$. Up to extending $\beta_i$, one can assume that $v_i = v$. In particular, the subgroups $(\Lambda_i)_{i \in \mathbb{N}}$ of $\mathbb{Z}^d$ are pairwise $\mathscr{H}$-equivalent w.r.t. $v$ so by Lemma \ref{eqrel}, there exists an $\mathscr{H}$-path $E_{1,i}, ..., E_{r_i,i}$ labeled $e_i, ..., \overline{e_{i+1}}$ that connects $V_i$ to $V_{i+1}$. Denoting by $\mathcal{G}$ the resulting $\mathscr{H}$-graph, Lemma \ref{forest} delivers an $\mathscr{H}$-graph $\mathcal{F}$ \begin{itemize}
			\item that contains $(K_i)_{i \in \mathbb{N}}$ as disjoint subgraphs;
			\item such that the quotient $\mathcal{F} / \bigsqcup_{i \in \mathbb{N}}K_i$ is a forest.
		\end{itemize}
		Hence, by Lemma \ref{completiond}, there exists a saturated $\mathscr{H}$-preaction $\beta$ that extends $\beta_i$ for every $i \in \mathbb{N}$ and whose $\mathscr{H}$-graph is $\mathcal{F}$. This proves that the conjugacy class of the subgroup of $G$ which is associated to $\beta$ is dense in $\mathcal{K}(G) \cap \Ph_{\mathscr{H},v}^{-1}(\pi_v(\Lambda))$. 
	\end{proof}
	
	Now we study the topology of the pieces of the partition \[\Sub(G) = \bigsqcup_{\Lambda \leq \mathbb{Z}^d}\Ph_{\mathscr{H},v}^{-1}(\pi_v(\Lambda)).\]
	
	\begin{proposition}\label{topopiece}
		For any $\Lambda_0 \leq \mathbb{Z}^d$:
		\begin{enumerate}
			\item if $\Lambda_0$ has finite index in $\mathbb{Z}^d$, then the fiber $\Ph_{\mathscr{H},v}^{-1}(\pi_v(\Lambda_0))$ is open;
			\item $\Ph_{\mathscr{H},v}^{-1}(\pi_v(\Lambda_0))$ is an $F_{\sigma}$;
			\item $\Ph_{\mathscr{H},v}^{-1}(\pi_v(\{\Lambda_0\}))$ is closed if and only if $\{\Lambda_1 \leq \mathbb{Z}^d \mid \Lambda_1 \simeq \Lambda_0 \}$ is finite. 
		\end{enumerate}
	\end{proposition}
	
	\begin{remark}
		In particular, $\Ph_{\mathscr{H},v}^{-1}(\pi_v(\{0\}))$ is closed. 
	\end{remark}
	
	\begin{remark}
		If the image of the modular homomorphism is trivial, then the third item of Lemma \ref{topopiece} together with Lemma \ref{commensurable} implies that all pieces are closed.
	\end{remark}
	
	\begin{proof}[Proof of Proposition \ref{topopiece}]
		Notice that we have \[\Ph_{\mathscr{H},v}^{-1}(\pi_v(\Lambda_0)) = \bigcup_{\Lambda_1 \simeq \Lambda_0}\left\{ H \leq G \mid H \cap G_v = \Lambda_1 \right\}.\]
		By Lemma \ref{topgen}, this is an $F_{\sigma}$ as a countable union of closed subsets of $\Sub(G)$. Hence, we get the second point.
		
		If $\Lambda_0$ has finite index in $\mathbb{Z}^d$, this is an open subset of $\Sub(G)$ as a union of open sets by Lemma \ref{topgen}. This proves the first point.
		
		If $\{\Lambda_1 \leq \mathbb{Z}^d \mid \Lambda_1 \simeq \Lambda_0 \}$ is finite, then $\Ph_{\mathscr{H},v}^{-1}(\pi_v(\Lambda_0))$ is closed as a finite union of closed subsets by Lemma \ref{topgen}. Otherwise, there exists a sequence of subgroups $(\Lambda_n)_{n\in \mathbb{N}} \in (\pi_v^{-1}(\pi_v(\Lambda_0)))^{\mathbb{N}}$ that converges to a subgroup $\Lambda \leq \mathbb{Z}^d$ whose rank is strictly less than $\rk(\Lambda_0)$. By a straightforward induction using Corollary \ref{coreq}, there exists an $\mathscr{H}$-path $E_1, ..., E_n, ...$ such that: for every $n \in \mathbb{N}$, there exists $k_n \in \mathbb{N}$ such that $\src\left(E_{k_n}\right)$ is labeled $(\Lambda_n, v)$. Thus, by Lemma \ref{forest}, there exists a saturated $\mathscr{H}$-graph which is a forest that contains $E_1, ..., E_n, ...$ as a sub-$\mathscr{H}$-graph. Hence by Lemma \ref{completiond}, there exists a subgroup $H \leq G$ and elements $g_n \in G$ such that $g_n H g_n^{-1} \cap G_v = \Lambda_n$ for every $n \in \mathbb{N}$. Up to extracting, the sequence $g_n H g_n^{-1}$ converges to a subgroup $K \leq G$ that satisfies $K \cap G_v = \Lambda$. In particular, as the rank is constant on the fibers of $\pi_v$, one has $K \notin \Ph_{\mathscr{H},v}^{-1}(\pi_v(\Lambda))$, which proves that $\Ph_{\mathscr{H},v}^{-1}(\pi_v(\Lambda_0))$ is not closed.
	\end{proof}
	
	\begin{remark}
	The definition of the equivalence relation $\simeq$ still makes sense if the graph $\mathscr{H}$ is infinite and the same proof extends to the class of groups that act (non necessarily cocompactly) on an oriented tree with vertex and edge stabilizers isomorphic to $\mathbb{Z}^d$. In the case of a non-cocompact action on the Bass-Serre tree, we thus obtain a dynamical partition of the whole set of subgroups by Remark \ref{rknoncpct}.
	\end{remark}
	
	\subsection{\texorpdfstring{Case where $G = \mathbb{Z}^d \rtimes \mathbb{F}_r$}{Case where $G = \mathbb{Z}^d \rtimes \mathbb{F}_r$}}\label{semidirect}
	
	Now we assume that $\mathscr{H}$ consists of a collection of $r$ loops $e_1,...,e_r$ based at a single vertex $v$ such that, for every $i \in \llbracket 1,r \rrbracket$, the label $(\rA_i,\B_i)$ of $e_i$ satisfies: $\rA_i \in \GL_d(\mathbb{Z})$ and $\B_i \in \GL_d(\mathbb{Z})$. Denoting by \[\rP_i := \rA_i\B_i^{-1},\]
	the group $G$ is isomorphic to the semidirect product of $G_v = \mathbb{Z}^d$ with the free group $\mathbb{F}_r = \langle a_1,...,a_r \rangle$ of rank $r$, where the generator $a_i$ acts on $\mathbb{Z}^d$ by multiplication by $\rP_i$. Let us denote by $\rho : \mathbb{F}_r \to \GL_d(\mathbb{Z})$ the homomorphism that satisfies $\rho(a_i) = \rP_i$ for every $i \in \llbracket 1, r \rrbracket$ and by $\Gamma = \rho(\mathbb{F}_r)$. Observe that this $\mathbb{F}_r$-action induces an $\mathbb{F}_r$-action on $\Sub(G_v)$ defined as follows: for any subgroup $\Lambda \leq G_v$ and any $\gamma \in \mathbb{F}_r$: \begin{align*}
		\gamma \cdot \Lambda &= \rho(\gamma)\Lambda \\
		&= (u, \gamma)\Lambda(u, \gamma)^{-1} \ \forall u \in \mathbb{Z}^d 
	\end{align*}
	(when identifying $G_v$ with $G_v \times \{1\}$ in $G_v \rtimes \mathbb{F}_r$).
	\begin{remark}
		If $\Lambda$ is a finite index subgroup of $G_v$, then for every $\gamma \in \mathbb{F}_r$, one has $|\det(\rho(\gamma)\Lambda)| = |\det(\Lambda)|$. As there exist only finitely many lattices of $\mathbb{Z}^d$ of a given determinant, the orbit $\mathbb{F}_r \cdot \Lambda = \{\rho(\gamma)\Lambda \mid \gamma \in \mathbb{F}_r\}$ is finite.
	\end{remark}
	
	Let us denote by $\pi : \mathbb{Z}^d \rtimes \mathbb{F}_r \to \mathbb{F}_r$ the projection. 
	
	Our goal is to decompose the perfect kernel of $G$ into countably many pieces on which the action by conjugation contains a dense orbit. Let us recall that (by Theorem \ref{kerneld}), one has \[\mathcal{K}(G) = \{H \leq G \mid \pi(H) \in \Sub_{[\infty]}(\mathbb{F}_r)\}.\]
	Notice that, as the subgroup $G_v$ is normal, denoting by $Conj(G_v)$ the set of classes for the action of $G$ on $\Sub(G_v)$ by conjugation, the following partition \begin{equation}\label{partition}\Sub(G) = \bigsqcup_{\mathscr{C} \in Conj(G_v)}\{H \leq G \mid H \cap G_v \in \mathscr{C} \}\end{equation} is $G$-invariant. Denoting by $\mathscr{C} = \{\rho(\gamma)\Lambda_0, \gamma \in \mathbb{F}_r\}$ for some $\Lambda_0 \leq G_v$, we get that \[\{H \leq G \mid H \cap G_v \in \mathscr{C} \} = \bigcup_{\gamma \in \mathbb{F}_r} \{H \leq G \mid H \cap G_v = \rho(\gamma)\Lambda_0\}.\] Notice that $|\det|$ is constant on $\mathscr{C}$ for every $\mathscr{C} \in Conj(G_v)$. In particular, there are infinitely many pieces in the decomposition \eqref{partition}.
	
	This partition leads to a $G$-invariant partition of the perfect kernel into countably many pieces \[\mathcal{K}(G) = \bigsqcup_{\mathscr{C} \in Conj(G_v)}\{H \in \mathcal{K}(G) \mid H \cap G_v \in \mathscr{C} \}.\]
    Notice that this is exactly the decomposition \eqref{dynpartperfker} we obtained in the previous case of a GBS group $G$ which is not $\mathbb{Z}^d$-by-free: two subgroups $\Lambda_0$, $\Lambda_1$ of $G_v$ can arise as the labels of two vertices of some connected $\mathscr{H}$-graph if and only if there exists some $\gamma \in \mathbb{F}_r$ such that $\rho(\gamma)\Lambda_0=\Lambda_1$, or equivalently, if and only if $\Lambda_0$ and $\Lambda_1$ belong to the same orbit under the $G$-conjugation. However, the proof of Theorem \ref{dynamic} does not extend to our new setting, because the relation $\simeq$ need not be transitive anymore. This difficulty turns out to be a real obstruction: the conjugation action need not be topologically transitive on each piece in our new setting. This comes from the fact that the skeleton of the $\mathscr{H}$-graph of a subgroup of $G$ is related to its intersection with the vertex group $G_v$ as follows: 
	\begin{lemma}\label{stabd}
		Let $H$ be a subgroup of $G$. Then $\pi(H) \leq \Stab_{\mathbb{F}_r}(H \cap G_v)$.
	\end{lemma}
	\begin{proof}
		Let $\gamma \in \pi(H)$. There exists $u \in \mathbb{Z}^d$ such that $(u, \gamma) \in H$. For any $(w, 1) \in H \cap G_v$, one has \begin{align*}
			(u, \gamma)(w,1)(u,\gamma)^{-1} &= (\rho(\gamma)(w),1) \\
			&\in H
		\end{align*}
		which implies that $\rho(\gamma)(w) \in H \cap G_v$. Consequently, $\rho(\gamma)(H\cap G_v) = H \cap G_v$ thus $\gamma \in \Stab_{\mathbb{F}_r}(H \cap G_v)$. Hence, $\pi(H) \leq \Stab_{\mathbb{F}_r}(H \cap G_v)$.
	\end{proof}
	
	Let us fix $\mathscr{C} \in Conj(G_v)$. Lemma \ref{stabd} allows us to decompose \begin{align*}\mathcal{P}_{\mathscr{C}} :&= \{H \in \mathcal{K}(G) \mid H \cap G_v \in \mathscr{C} \} \\ &= \{H \in \mathcal{K}(G) \mid H\cap G_v \in \mathscr{C} \text{ \ and \ } \pi(H) \in \mathcal{K}(\Stab_{\mathbb{F}_r}(H \cap G_v))  \} \\ &\bigsqcup \{H \in \mathcal{K}(G) \mid H \cap G_v \in \mathscr{C} \text{ \ and \ } \pi(H) \notin \mathcal{K}(\Stab_{\mathbb{F}_r}(H \cap G_v))\},  \end{align*} each of these two pieces being invariant under $G$-conjugation.
	Notice that the second piece \[\mathcal{D}_{\mathscr{C}} = \{H \in \mathcal{K}(G) \mid H \cap G_v \in \mathscr{C} \text{ \ and \ } \pi(H) \notin \mathcal{K}(\Stab_{\mathbb{F}_r}(H \cap G_v))\}\] of this last decomposition is always countable and open for the induced topology on $\mathcal{P}_{\mathscr{C}}$. 
	\begin{remark}\label{disjonction}
	More precisely, two cases can occur: \begin{itemize}
		\item If $\Stab_{\mathbb{F}_r}(\Lambda_0)$ is infinitely generated or $\Stab_{\mathbb{F}_r}(\Lambda_0)$ has finite index in $\mathbb{F}_r$ (for some, equivalently for all $\Lambda_0 \in \mathscr{C}$), then $\mathcal{D}_{\mathscr{C}}$ is empty;
		\item Otherwise, it consists of \[\{H \in \mathcal{K}(G) \mid H \cap G_v \in \mathscr{C} \text{ \ and \ } \pi(H) \notin \Sub_{[\infty]}(\Stab_{\mathbb{F}_r}(H \cap G_v))\}\] if $\Stab_{\mathbb{F}_r}(\Lambda_0)$ is not infinite cyclic, and of $\mathcal{P}_{\mathscr{C}}$ otherwise.
	\end{itemize}
	\end{remark}
	\begin{lemma}\label{denseorbitsemidirect}
		For every $\mathscr{C} \in Conj(G_v)$, there exists a dense orbit in \[\mathcal{P}_{\mathscr{C}} \setminus \mathcal{D}_{\mathscr{C}} = \{H \in \mathcal{K}(G) \mid H\cap G_v \in \mathscr{C} \text{ \ and \ } \pi(H) \in \mathcal{K}(\Stab_{\mathbb{F}_r}(H \cap G_v))  \}.\] 
	\end{lemma}
	
	To prove this lemma we will use the formalism of $\mathscr{H}$-graphs. Notice that in this context, the $\mathscr{H}$-graph $\mathcal{G}$ of a subgroup $H \leq G$ will be uniquely determined by $\pi(H)$ and $H \cap G_v$. It is the Schreier graph of the subgroup $\pi(H) \leq \mathbb{F}_r$ (with respect to the generating set $\{a_1,..., a_r\}$) whose labels are defined as follows: if $V_0$ is a vertex labeled $\Lambda_0 := H \cap G_v$, and $V$ is any other vertex of $\mathcal{G}$, then, denoting by $E_1,...,E_s$ an edge path labeled $f_1,...,f_s$ that connects $V_0$ to $V$, the Transfer Equation \ref{transfer} tells us that the label of $V$ is the subgroup $\left(\M_{f_s, \trg}\M_{f_s, \src}^{-1} \ldots \M_{f_1, \trg}\M_{f_1, \src}^{-1} \right)\Lambda_0$ of $\mathbb{Z}^d$. In particular, the set of labels of the vertices of $\mathcal{G}$ is exactly $\{\rho(\gamma)\Lambda_0, \gamma \in \mathbb{F}_r\}$. 
	
	\begin{proof}
		Let $\mathscr{C} \in Conj(G_v)$ and let $\Lambda_0 \in \mathscr{C}$. Let us denote by $\Gamma_0 := \Stab_{\mathbb{F}_r}(\Lambda_0)$. Let $(H_i)_{i \in \mathbb{N}^*} \in (\mathcal{P}_{\mathscr{C}} \setminus \mathcal{D}_{\mathscr{C}})^{\mathbb{N}}$ be the sequence of finitely generated subgroups such that $H_i \cap G_v = \Lambda_0$ for every $i \in \mathbb{N}^*$. For every $i \in \mathbb{N}^*$, let $\alpha_i$ be the saturated $\mathscr{H}$-preaction associated to $H_i$, and let $\mathcal{G}_i$ be its $\mathscr{H}$-graph (based at a vertex $V_i$). Let $\beta_i$ be a sub-$\mathscr{H}$-preaction of $\alpha_i$ whose stabilizer is $H_i$ and whose $\mathscr{H}$-graph $K_i$ is a finite subgraph of $\mathcal{G}_i$ that contains $V_i$ (legit, because $H_i$ is finitely generated). Let $S_i$ be the graph obtained by forgetting the labels of the vertices of $\mathcal{G}_i$, pointed at the vertex $V_i$. The graph $S_i$ is the Schreier graph of $\pi(H_i)$ with respect to the generating set $\{a_1,...a_r\}$ of $\mathbb{F}_r$. As $\pi(H_i)$ is a subgroup of $\Gamma_0$, denoting by $S_0$ the Schreier graph of $\Gamma_0$, the graph $S_i$ is in fact a covering of the graph $S_0$ (for every $i \in \mathbb{N}^*$). As $\pi(H_i)$ belongs to $\mathcal{K}(\Gamma_0)$, one has the following dichotomy:
		\begin{itemize}
			\item either $\Gamma_0$ is not finitely generated;
			\item or $\Gamma_0$ is finitely generated and the covering map $S_i \to S_0$ has infinite degree for every $i \in \mathbb{N}^*$.
		\end{itemize} 
		Let us define $F_i$ as the subgraph of $S_i$ obtained by forgetting the labels of the subgraph $K_i$ of $\mathcal{G}_i$. In both cases, Remark \ref{ideafreegp} provides a covering map $S \to S_0$ such that $S$ contains $(F_i)_{i \in \mathbb{N}^*}$ as disjoint subgraphs and such that the quotient $S/\bigsqcup_{i \in \mathbb{N}^*}F_i$ is a tree. This covering corresponds to an infinite index subgroup $\Gamma \in \Sub_{[\infty]}(\Gamma_0)$. After labeling the vertex $V_1$ by the subgroup $\Lambda_0$ (and all the other vertices of $S$ so that the Transfer Equation \ref{transfer} is satisfied), we obtain an infinite $\mathscr{H}$-graph $\mathcal{F}$: \begin{itemize}
			\item that contains $(K_i)_{i \in \mathbb{N}^*}$ as disjoint subgraphs;
			\item such that the quotient $\mathcal{F}/\bigsqcup_{i \in \mathbb{N}^*}K_i$ is a tree.
		\end{itemize}
		Thus, by Lemma \ref{completiond}, there exists a saturated $\mathscr{H}$-preaction $\gamma$ on a pointed countable set $(X,x)$ that extends $\beta_i$ for every $i \in \mathbb{N}^*$ and whose $\mathscr{H}$-graph $\mathcal{F}$ is infinite. Denoting by $L$ the stabilizer of $\gamma$, as the $\mathscr{H}$-graph of $L$ contains a vertex labeled $\Lambda_0$, one moreover has $L \cap G_v \in \mathscr{C}$. Thus, $L \in \mathcal{P}_{\mathscr{C}}$. As the covering map $S \to S_0$ has infinite degree, one also has $\pi(L) \in \Sub_{[\infty]}(\mathbb{F}_r) \cap \mathcal{K}(\Stab_{\mathbb{F}_r}(L \cap G_v))$, thus $L \in \mathcal{P}_{\mathscr{C}} \smallsetminus \mathcal{D}_{\mathscr{C}}$. The fact that $\gamma$ extends all the $\mathscr{H}$-preactions $\beta_i$ implies that the conjugacy class of $L$ accumulates on all the finitely generated subgroups of $\mathcal{P}_{\mathscr{C}} \smallsetminus \mathcal{D}_{\mathscr{C}}$ whose intersection with $G_v$ is $\Lambda_0$. For any other element $K$ of $\mathcal{P}_{\mathscr{C}} \smallsetminus \mathcal{D}_{\mathscr{C}}$, there exists a conjugate $gKg^{-1}$ whose intersection with $G_v$ is $\Lambda_0$. Writing $gKg^{-1}$ as the limit of an increasing sequence of finitely generated subgroups whose intersection with $G_v$ is $\Lambda_0$, we finally get that $gKg^{-1}$ is the limit of a sequence of conjugates of $L$. Thus, $K$ is itself the limit of a sequence of conjugates of $L$, so the orbit of $L$ under the conjugation action is dense in $\mathcal{P}_{\mathscr{C}} \smallsetminus \mathcal{D}_{\mathscr{C}}$. 
	\end{proof} 
	
	Hence we obtain the following decomposition of $\mathcal{K}(G)$:
	\begin{theorem}\label{dynsemidir}
		There exists a $G$-invariant countable partition \[\mathcal{K}(G) = \bigsqcup_{\mathscr{C} \in Conj(G_v)}\mathcal{P}_{\mathscr{C}}\] into $F_{\sigma}$-subsets of $\mathcal{K}(G)$, and, for every $\mathscr{C} \in Conj(G_v)$, a countable $G$-invariant open subset $\mathcal{D}_{\mathscr{C}} \subseteq \mathcal{P}_{\mathscr{C}}$ (for the induced topology on $\mathcal{P}_{\mathscr{C}}$) such that there exists a dense orbit in $\mathcal{P}_{\mathscr{C}} \setminus \mathcal{D}_{\mathscr{C}}$. Moreover, denoting by $\mathscr{C} = \{\rho(\gamma)\Lambda_0, \gamma \in \mathbb{F}_r\}$ for some $\Lambda_0 \leq \mathbb{Z}^d$:   
		\begin{enumerate}
			\item if $\Lambda_0$ has finite index in $\mathbb{Z}^d$, then $\mathcal{P}_{\mathscr{C}}$ is a clopen set;
			\item $\mathcal{P}_{\mathscr{C}}$ is closed if and only if $\Stab_{\mathbb{F}_r}(\Lambda_0)$ has finite index in $\mathbb{F}_r$;
			\item if $\Stab_{\mathbb{F}_r}(\Lambda_0)$ either is infinitely generated or has finite index in $\mathbb{F}_r$, then $\mathcal{D}_{\mathscr{C}} = \emptyset$.
		\end{enumerate}
	\end{theorem}
	
\begin{proof}
    The existence of a dense orbit in $\mathcal{P}_{\mathscr{C}} \smallsetminus \mathcal{D}_{\mathscr{C}}$ is provided by Lemma \ref{denseorbitsemidirect}. The fact that $\mathcal{P}_{\mathscr{C}}$ is an $F_{\sigma}$ results from Lemma \ref{topgen}.
    
    If $\Lambda_0$ has finite index in $\mathbb{Z}^d$, then $\mathscr{C}$ is finite (because it is a set of finite index subgroups of prescribed finite index). Thus, $\mathcal{P}_{\mathscr{C}} = \bigsqcup_{C \in \mathscr{C}}\{H \in \mathcal{K}(G) \mid H \cap G_v = C \}$ is a finite union of clopen sets by Lemma~\ref{topgen}, so is itself clopen. This proves the first item.
    
    Now let us turn to the proof of the second item. Let $\Lambda_0 \leq \mathbb{Z}^d$ and let $\mathscr{C}$ be the $G$-conjugacy class of $\Lambda_0$. The fact that $\Stab_{\mathbb{F}_r}(\Lambda_0)$ has finite index in $\mathbb{F}_r$ is equivalent to the finiteness of the orbit $\{\rho(g)\Lambda_0, g \in G\}$. In particular, if $\Stab_{\mathbb{F}_r}(\Lambda_0)$ has finite index in $\mathbb{F}_r$, then \[\mathcal{P}_{\mathscr{C}} = \bigcup_{g \in \mathbb{F}_r}\{H \in \mathcal{K}(G) \mid H \cap G_v = \rho(g)\Lambda_0\}\] is closed as a finite union of closed sets by Lemma \ref{topgen}. Conversely, let us assume that the set $\{\rho(g)\Lambda_0, g \in G\}$ is infinite. In particular, there exists a sequence $(g_n)_{n \in \mathbb{N}} \in \mathbb{F}_r^{\mathbb{N}}$ such that $\rho(g_n)\Lambda_0$ converges to a subgroup $\Lambda \leq G_v$ of rank strictly less than the one of $\Lambda_0$. In particular, $\rho(g_n)\Lambda_0 \in \mathcal{P}_{\mathscr{C}}$ for every $n$, but $\Lambda \notin \mathcal{P}_{\mathscr{C}}$ which implies that $\mathcal{P}_{\mathscr{C}}$ is not closed. 
    
    The third item results from Remark \ref{disjonction}.
\end{proof}

\begin{remark}
Again, we did not use the cocompactness of the action of $G$ on its Bass-Serre tree. In other words, Theorem \ref{dynsemidir} extends to the case where $r=\infty$.
\end{remark}

To conclude, we give an explicit example where the set $\mathcal{D}_{\mathscr{C}}$ is non-empty, and where the conjugation action has no dense orbit for the induced topology on $\mathcal{P}_{\mathscr{C}}$: 
\begin{example}
    Let $\rP_1$ and $\rP_2$ be two matrices that generate a free subgroup of $\SL_2(\mathbb{Z})$; for instance,  $\rP_1=\begin{pmatrix}
    1 & 2 \\
    0 & 1
  \end{pmatrix}$ and $\rP_2 =  \begin{pmatrix}
    1 & 0 \\
    2 & 1
  \end{pmatrix}$. Let $G$ be the semidirect product $\mathbb{Z}^2 \rtimes \mathbb{F}_2$, where, denoting by $\{a_1, a_2\}$ a basis of $\mathbb{F}_2$, the generator $a_i$ acts by multiplication by $\rP_i$ on $\mathbb{Z}^2$. 
  
  Let $\Lambda_0= \begin{pmatrix}
    1  \\
    0 
  \end{pmatrix}\mathbb{Z}$. Then, any matrix of $\SL_2(\mathbb{Z})$ stabilizing the subgroup $\Lambda_0$ of $\mathbb{Z}^2$ is a power of $\begin{pmatrix}
    1 & 1 \\
    0 & 1
  \end{pmatrix}$. As the homomorphism $\rho: \begin{array}{ccccc}
 & \mathbb{F}_2 & \to & \SL_2(\mathbb{Z}) \\
& a_1 & \mapsto & \rP_1 \\
& a_2 & \mapsto & \rP_2
\end{array}$ is injective, the stabilizer $\Stab_{\mathbb{F}_2}(\Lambda_0)$ is also infinite cyclic (in fact, it is the subgroup generated by $a_1$). In particular, its perfect kernel is empty, so denoting by $\mathscr{C}:=\mathbb{F}_2 \cdot \Lambda_0$ the $G$-conjugacy class of $\Lambda_0$, one has \[\mathcal{D}_{\mathscr{C}} = \mathcal{P}_{\mathscr{C}}.\]

Let us show that there is no dense orbit for the conjugation action $G \curvearrowright \mathcal{P}_{\mathscr{C}}$. For any $v \in \mathbb{Z}^2$, let us consider the following invariant subsets \[U_v = \{H \in \mathcal{P}_{\mathscr{C}} \mid \exists g \in G, g(v,a_1)g^{-1} \in H\}.\]
The set $U_v$ is open as a union of clopen sets: \[U_v = \bigcup_{g \in G}\{H \in \mathcal{P}_{\mathscr{C}} \mid g(v,a_1)g^{-1} \in H\}.\]

\begin{claim}
For every $v,w \in \mathbb{Z}^2$ such that $v-w \notin \Lambda_0$, one has $U_v \cap U_w = \emptyset$. 
\end{claim}

As the quotient $\mathbb{Z}^2/\Lambda_0$ is infinite, this will imply the existence of an infinite countable set of pairwise disjoint $G$-invariant open subsets, making the existence of a dense orbit impossible.

\begin{cproof}
By contraposition, let $v, w \in \mathbb{Z}^2$ such that $U_v \cap U_w \neq \emptyset$. Let $H \in \mathcal{P}_{\mathscr{C}}$. Up to conjugating $H$, there exists $g \in G$ such that \begin{itemize}
    \item $(v,a_1) \in H$;
    \item $g(w,a_1)g^{-1} \in H$. 
\end{itemize}
As $H \in \mathcal{P}_{\mathscr{C}}$, there exists $\gamma \in \mathbb{F}_2$ such that $H \cap \mathbb{Z}^2 = \rho(\gamma) \Lambda_0$. Thus, $\Stab_{\mathbb{F}_2}(H \cap \mathbb{Z}^2) = \gamma \langle a_1 \rangle \gamma^{-1}$. Denoting by $g = (u_0,\gamma_0)$, one deduces that $a_1, \gamma_0 a_1 \gamma_0^{-1} \in \gamma \langle a_1 \rangle \gamma^{-1}$. This forces $\gamma$ and $\gamma_0$ to belong to $\langle a_1 \rangle$. In particular, $H \cap \mathbb{Z}^2 = \Lambda_0$. Let $k \in \mathbb{Z}$ such that $\gamma_0 = a_1^k$. One has \begin{align*}
    g(w,a_1)g^{-1}(v,a_1)^{-1} &= \left(u_0, a_1^k\right)(w,a_1)\left(u_0,a_1^{k}\right)^{-1}(v,a_1)^{-1} \\
    &= \left(u_0 + \rho(a_1)^kw, a_1^{k+1}\right)\left(-\rho(a_1)^{-k}u_0, a_1^{-k}\right)(v,a_1)^{-1} \\
    &= (u_0 + \rho(a_1)^kw - \rho(a_1)u_0, a_1)\left(-\rho(a_1)^{-1}v,a_1^{-1}\right) \\ 
    &= \left(u_0 - \rP_1u_0 + \rP_1^kw - v,1\right) \\
    &\in \left(H\cap \mathbb{Z}^2\right) \times \{1\} \\
    &= \Lambda_0 \times \{1\}.
\end{align*} 
Thus, $u_0 - \rP_1u_0 + \rP_1^kw - v \in \Lambda_0$. Notice that $(I_2 - \rP_1)\mathbb{Z}^2 \subseteq \Lambda_0$. Thus, \[u_0-\rP_1u_0 \in \Lambda_0, \] and \begin{align*}\rP_1^kw &= w +  \sum_{i=0}^{k-1}\rP_1^{i}(\rP_1w - w) \\
&\in w + \Lambda_0,
\end{align*} which implies that $w-v \in \Lambda_0$. 
\end{cproof}

\end{example}

\bibliographystyle{alpha}
\bibliography{reference}

@misc{solitar2,
      title={On the space of subgroups of {B}aumslag-{S}olitar groups {II}: {H}igh transitivity}, 
      author={Gaboriau, Damien and Le Maître, François and Stalder, Yves},
      year={2024},
      eprint={2410.23224},
      archivePrefix={arXiv},
      primaryClass={math.GR},
      url={https://arxiv.org/abs/2410.23224}, 
      note= {arXiv: 2410.23224. To appear in \textit{Confluentes. Math.}}
}

@misc{bontemps,
      title={Perfect kernel of generalized {B}aumslag-{S}olitar groups}, 
      author={Bontemps, Sasha},
      year={2026},
      eprint={2411.03221},
      archivePrefix={arXiv},
      primaryClass={math.GR},
      url={https://arxiv.org/abs/2411.03221}, 
      note={arXiv: 2411.03221. To appear in \textit{Transactions of the AMS}.}
}

@article{modular, title={Separability properties of higher rank {GBS} groups}, volume={57}, ISSN={1469-2120}, DOI={10.1112/BLMS.70024}, number={4}, journal={Bulletin of the London Mathematical Society}, publisher={John Wiley and Sons Ltd}, author={Lopez de Gamiz Zearra, Jone and Shepherd, Sam}, year={2025}, pages={1171--1194} }

@misc{measureeq,
      title={Measure equivalence classification of {B}aumslag-{S}olitar groups}, 
      author={Damien Gaboriau and Antoine Poulin and Anush Tserunyan and Robin Tucker-Drob and Konrad Wr\'obel},
    note={in preparation}
}

@book{serre,
  title={Arbres, amalgames, SL2: cours au Coll{\`e}ge de France},
  author={Serre, Jean-Pierre},
  series={Ast{\'e}risque / Soci{\'e}t{\'e} Math{\'e}matique de France},
  url={https://books.google.fr/books?id=jhsYwQEACAAJ},
  year={1983}
}

@Book{Kechris,
 Author = {Kechris, Alexander S.},
 Title = {Classical descriptive set theory},
 FSeries = {Graduate Texts in Mathematics},
 Series = {Grad. Texts Math.},
 ISSN = {0072-5285},
 Volume = {156},
 ISBN = {3-540-94374-9},
 Year = {1995},
 Publisher = {Berlin: Springer-Verlag},
 Language = {English},
 zbMATH = {722611},
 Zbl = {0819.04002}
}

@article{totipotent,
      title={On dense totipotent free subgroups in full groups}, 
      author={Carderi, Alessandro and Gaboriau, Damien and Le Maître, François},
year={2023},
journal={Geometry \& Topology },
pages={2297--2318},
volume={27}
}

@article{AG,
title = {Perfect kernel and dynamics: from {B}ass-{S}erre theory to hyperbolic groups},
author = {Azuelos, Pénélope and Gaboriau, Damien},
year = {2024},
doi = {10.1007/s00208-024-03038-w},
language = {English},
volume = {391},
pages = {4733–4789},
journal = {Mathematische Annalen},
issn = {0025-5831},
publisher = {Springer, New York, NY},
number = {3},

}

@article{hightrans,
	journal={Discrete Analysis},
	doi={10.19086/da.37645},
	title={A characterization of high transitivity for groups acting on trees},
	author={Fima, Pierre and Le Maître, François and Moon, Soyoung and Stalder, Yves},
	date={2022-08-31},
	year=2022,
}

@article{tworelators,
     author = {Baumslag , Gilbert and Solitar, Donald},
     title = {Some two-generator one-relator non-{H}opfian groups},
     journal = {Bull. Amer. Math. Soc.},
     volume = {68},
     number = {6},
     year = {1962},
     pages = { 199-201},
     language = {en},
     url = {http://dml.mathdoc.fr/item/1183524561}
}

@article{bass,
title = {Covering theory for graphs of groups},
journal = {Journal of Pure and Applied Algebra},
volume = {89},
number = {1},
pages = {3-47},
year = {1993},
issn = {0022-4049},
doi = {https://doi.org/10.1016/0022-4049(93)90085-8},
url = {https://www.sciencedirect.com/science/article/pii/0022404993900858},
author = {Bass, Hyman}
}

@article {whyte,
    AUTHOR = {Whyte, Kevin},
     TITLE = {The large scale geometry of the higher {B}aumslag-{S}olitar
              groups},
   JOURNAL = {Geom. Funct. Anal.},
  FJOURNAL = {Geometric and Functional Analysis},
    VOLUME = {11},
      YEAR = {2001},
    NUMBER = {6},
     PAGES = {1327--1343},
      ISSN = {1016-443X},
   MRCLASS = {20F65 (20F69)},
  MRNUMBER = {1878322},
MRREVIEWER = {Lee Mosher},
       DOI = {10.1007/s00039-001-8232-6},
       URL = {https://ezproxy-prd.bodleian.ox.ac.uk:2095/10.1007/s00039-001-8232-6},
}

@article{levitt,
     author = {Levitt, Gilbert},
     title = {Generalized {Baumslag{\textendash}Solitar} groups: rank and finite index subgroups},
     journal = {Annales de l'Institut Fourier},
     pages = {725--762},
     publisher = {Association des Annales de l{\textquoteright}institut Fourier},
     volume = {65},
     number = {2},
     year = {2015},
     doi = {10.5802/aif.2943},
     language = {en},
     url = {https://aif.centre-mersenne.org/articles/10.5802/aif.2943/}
}

@article{levittisom,
author = {Levitt, Gilbert},
year = {2007},
pages = {473--515},
title = {On the automorphism group of generalized {B}aumslag-{S}olitar groups},
volume = {11},
number = {1},
journal = {Geometry \& Topology},
doi = {10.2140/gt.2007.11.473}
}

@article{abelian,
author = {Cornulier, Yves and Guyot, Luc and Pitsch, Wolfgang},
year = {2010},
pages = {727--746},
title = {The space of subgroups of an abelian group},
volume = {81},
number = {3},
journal = {Journal of the London Mathematical Society},
doi = {10.1112/jlms/jdq016}
}

@article{lamplighter,
author = {Bowen, Lewis and Grigorchuk, Rostislav and Kravchenko, Rostyslav},
year = {2012},
pages = {763--782},
title = {Invariant random subgroups of lamplighter groups},
volume = {207},
journal = {Israel Journal of Mathematics},
doi = {10.1007/s11856-015-1160-1}
}

@article{solitar1,
author = {Carderi, Alessandro and Gaboriau, Damien and Le Maître, François and Stalder, Yves},
year = {2025},
pages = {1711--1758},
title = {On the space of subgroups of Baumslag–Solitar groups {I}: Perfect kernel and phenotype},
volume = {41},
number = {5},
journal = {Revista Matemática Iberoamericana},
doi = {10.4171/rmi/1549}
}

@article{button, title={Generalised Baumslag-Solitar groups and hierarchically hyperbolic groups}, volume={25}, number={4}, journal={Algebraic \& Geometric Topology}, publisher={Mathematical Sciences Publishers}, author={Button, Jack O.}, year={2025}, month={Aug}, pages={2253--2279} }

\bigskip
{\footnotesize
	
	\noindent
	{\textsc{ENS-Lyon,
			Unité de Mathématiques Pures et Appliquées,  69007 Lyon, France}}
	\par\nopagebreak \texttt{sasha.bontemps@ens-lyon.fr}
}

\end{document}